\documentclass[oneside,english]{amsart}
\usepackage[T1]{fontenc}
\usepackage[latin9]{inputenc}
\usepackage{bm}
\usepackage{amsthm}
\usepackage{amstext}
\usepackage{amssymb}
\usepackage[all]{xy}

\makeatletter
\numberwithin{equation}{section}
\numberwithin{figure}{section}
\theoremstyle{plain}
\newtheorem{thm}{\protect\theoremname}
  \theoremstyle{definition}
  \newtheorem{defn}[thm]{\protect\definitionname}
  \theoremstyle{plain}
  \newtheorem{lem}[thm]{\protect\lemmaname}
  \theoremstyle{plain}
  \newtheorem{prop}[thm]{\protect\propositionname}

\makeatother

  \providecommand{\definitionname}{Definition}
  \providecommand{\lemmaname}{Lemma}
  \providecommand{\propositionname}{Proposition}
\providecommand{\theoremname}{Theorem}

\begin{document}

\title{A proof of the refined class number formula of Gross}

\author{Minoru Hirose}

\address{Department of Mathematics, Kyoto University, Kitashirakawa Oiwake-cho,
Sakyo-ku, Kyoto, 606-8502, Japan}

\email{hirose@math.kyoto-u.ac.jp}

\date{\today}
\begin{abstract}
In 1988, Gross proposed a conjectural congruence between Stickelberger
elements and algebraic regulators, which is often referred to as the
refined class number formula. In this paper, we prove this congruence.
\end{abstract}
\maketitle
\tableofcontents{}

\section{Introduction\label{sec:intro}}

\subsection{\label{sub:introGC}The conjecture of Gross}

Let $F$ be a totally real field and $K$ a finite abelian extension
of $F$. Let $S$ and $T$ be disjoint finite sets of places of $F$
such that $S$ contains all the infinite places and all places ramifying
in $K/F$. We assume that
\[
\ker(\mu_{K}\to\prod_{\mathfrak{p}\in T_{_{K}}}(\mathcal{O}_{K}/\mathfrak{p})^{\times})=\{1\}
\]
where $\mu_{K}$ is the group of roots of unity in $K$ and $T_{K}$
is the set of places of $K$ lying above places in $T$. Put $G:={\rm Gal}(K/F)$.
Let $G_{v}\subset G$ be the decomposition group at $v\in S$. Let
$\Theta_{S,T,K}\in\mathbb{Z}[G]$ be the Stickelberger element. Define
the $(S,T)$-ideal class group by
\[
{\rm Cl}_{S,T}:={\rm coker}(F_{(T)}^{\times}\xrightarrow{\bigoplus{\rm ord}_{w}}\bigoplus_{w\notin S\cup T}\mathbb{Z})
\]
where $F_{(T)}^{\times}$ is the group of elements of $F^{\times}$
which are congruent to $1$ modulo all places in $T$. Put $h_{S,T}:=\#{\rm Cl}_{S,T}$.
We write $\mathcal{O}_{S,T}^{\times}$ for the group of $S$-units
of $F$ which are congruent to $1$ modulo all places in $T$. Put
$r:=\#S-1$ and $S=\{v_{0},\dots,v_{r}\}$. From the condition, $\mathcal{O}_{S,T}^{\times}$
is a free abelian group of rank $r$. Take a $\mathbb{Z}$-basis $\left\langle u_{1},\dots,u_{r}\right\rangle $
of $\mathcal{O}_{S,T}^{\times}$ such that $(-1)^{\#T}\det(-\log\left|u_{i}\right|_{v_{j}})_{1\leq i,j\leq r}>0$.
Put $I:=\ker(\mathbb{Z}[G]\to\mathbb{Z})$. Define $R_{G,S,T}\in I^{r}/I^{r+1}$
by
\[
R_{G,S,T}:=-h_{S,T}\det({\rm rec}_{v_{i}}(u_{j})-1)_{1\leq i,j\leq r}
\]
where ${\rm rec}_{v}$ is the composite map $F^{\times}\to F_{v}^{\times}\to G_{v}\subset G$.
In this paper, we prove the following congruence.
\begin{thm}
[conjectured by Gross in {\cite[Conjecture 4.1]{MR931448}}]\label{Thm:main1}
\[
\Theta_{S,T,K}\equiv R_{G,S,T}\pmod{I^{r+1}}.
\]

\end{thm}
This congruence is an analogue of $(S,T)$-version of Dedekind's class
number formula
\[
\lim_{s\to0}s^{-r}\zeta_{F,S,T}(s)=-h_{S,T}\det(-\log\left|u_{i}\right|_{v_{j}})_{1\leq i,j\leq r}.
\]
Let $H$ be the maximal abelian unramified extension of $F$ such
that all the places in $S$ split completely at $H$. Assume that
$H\subset K$. Put $I_{G_{v}}:=\ker(\mathbb{Z}[G]\to\mathbb{Z}[G/G_{v}])$,
$I_{G_{H}}:=\ker(\mathbb{Z}[G]\to\mathbb{Z}[{\rm Gal}(H/F)])$, and
$n_{S,T}:=-\frac{h_{S,T}}{\#{\rm Gal}(H/F)}\in\mathbb{Z}$. We lift
$R_{G,S,T}$ to $(\prod_{v\in S\setminus\{v_{0}\}}I_{G_{v}})/(I_{G_{H}}\prod_{v\in S\setminus\{v_{0}\}}I_{G_{v}})$
by
\[
R_{G,S,T}:=n_{S,T}\sum_{c\in{\rm Gal}(H/F)}[c]\det({\rm rec}_{v_{i}}(u_{j})-1)_{1\leq i,j\leq r}.
\]
 In fact, we prove the following stronger claim.
\begin{thm}
\label{Thm:main2}
\[
\Theta_{S,T,K}\equiv R_{G,S,T}\pmod{I_{G_{H}}\prod_{v\in S\setminus\{v_{0}\}}I_{G_{v}}}.
\]

\end{thm}
Throughout this paper, we assume that $F\neq\mathbb{Q}$ since it
is already known that Theorem \ref{Thm:main2} holds for $F=\mathbb{Q}$
\cite{MR1104783}\cite{MR2336038}.

\subsection{\label{sub:introEGC}The enhancement of the conjecture of Gross}

For a place $\mathfrak{p}$ of $F$, we denote by ${\rm ch}(\mathfrak{p})$
the residue characteristic of $\mathfrak{p}$. Let us consider the
case of $T=\{\mathfrak{q}\}$ where $\mathfrak{q}$ is the primed
ideal of $F$ such that ${\rm ch}(\mathfrak{q})\geq[F:\mathbb{Q}]+2$.
For a finite place $v$ not in $S$, we put $J_{v}:=O_{v}^{\times}$
where $O_{v}$ is the maximal compact subring of $F_{v}$. For a finite
place $v$ in $S$, we fix an enough small open subgroup $J_{v}$
of $O_{v}^{\times}$ such that $J_{v}\subset\ker(F_{v}^{\times}\xrightarrow{{\rm rec}_{v}}G_{v})$.
For an infinite place $v$ of $F$, we denote by $J_{v}$ the identity
component of $F_{v}^{\times}\simeq\mathbb{R}^{\times}$. We put $N_{v}:=F_{v}^{\times}/J_{v}$,
$N_{F}:=\mathbb{A}_{F}^{\times}/\prod_{v}J_{v}$, and $N^{S}:=\prod'_{v\notin S}N_{v}$.
For $v\in S$, we put $I_{v}:=\ker(\mathbb{Z}[N_{F}]\to\mathbb{Z}[N_{F}/N_{v}])$.
We write ${\rm rec}$ for any homomorphism induced from the reciprocity
map. We put $I_{H}:=\ker(\mathbb{Z}[N_{F}]\xrightarrow{{\rm rec}}\mathbb{Z}[{\rm Gal}(H/F)])$.
Take a $\mathbb{Z}$-basis $\left\langle u_{1},\dots,u_{r}\right\rangle $
of $\mathcal{O}_{S,\{\mathfrak{q}\}}^{\times}$ such that $-\det(-\log\left|u_{i}\right|_{v_{j}})_{1\leq i,j\leq r}>0$.
Define $\hat{R}_{\mathfrak{q}}\in I_{v_{1}}\cdots I_{v_{r}}/I_{H}I_{v_{1}}\cdots I_{v_{r}}$
by
\[
\hat{R}_{\mathfrak{q}}:=n_{S,\{\mathfrak{q}\}}\sum_{c\in N^{S}/F^{\times}}[c]\det(f_{v_{i}}(u_{j})-1)_{1\leq i,j\leq r}
\]
where $f_{v}$ is the composite map $F^{\times}\to F_{v}^{\times}\to N_{v}\to N_{F}.$
Then we have ${\rm rec}(\hat{R}_{\mathfrak{q}})=R_{G,S,\{\mathfrak{q}\}}$.
We put $I_{F^{\times}}:=\ker(\mathbb{Z}[N_{F}]\to\mathbb{Z}[N_{F}/F^{\times}])$.
In this paper, we construct an element $\hat{\Theta}_{S,\mathfrak{q},S\setminus\{v_{0}\}}$
of $(\prod_{v\in S\setminus v_{0}}I_{v})/(I_{F^{\times}}\prod_{v\in S\setminus v_{0}}I_{v})$
which is mapped to $\Theta_{S,\{\mathfrak{q}\},K}$ by ${\rm rec}$.
The following theorem is an enhancement of the $T=\{\mathfrak{q}\}$
case of Theorem \ref{Thm:main2}.
\begin{thm}
\label{thm:main_hat}We have
\[
\hat{\Theta}_{S,\mathfrak{q},S\setminus\{v_{0})}\equiv\hat{R}_{\mathfrak{q}}\pmod{I_{H}\prod_{v\in S\setminus\{v_{0}\}}I_{v}}.
\]

\end{thm}
The most part of the proofs of Theorem \ref{Thm:main2} and Theorem
\ref{thm:main_hat} overlaps. Theorem \ref{thm:main_hat} is also
important as a special case of more generalized conjecture.

\subsection{Shintani data}

In this paper, we introduce the notion of Shintani data, which plays
an important role in the proofs of \ref{Thm:main2} and \ref{thm:main_hat}.
Fix $F,K,S,(J_{v})_{v}$ as in the previous two subsections, and fix
a prime ideal $\mathfrak{q}$ of $F$ such that $q\notin S$ and ${\rm ch}(\mathfrak{q})\neq2$.
Let $V$ be a subset of $S$. In this paper, we define the category
of Shintani data on $V$. A Shintani datum on $V$ is a quadruple
$(\mathcal{B},\mathcal{L},\vartheta,m)$ satisfying certain conditions.
We call $m\in\mathbb{Z}_{>0}$ an integrality of this Shintani datum.
For a subset $V'$ of $V$, we can naturally define a functor $|_{V'}$
from the category of Shintani data on $V$ to that on $V'$. For each
proper subset $V$ of $S$ and a Shintani datum ${\rm Sh}$ on $V$,
we construct a special element $Q^{N}({\rm Sh})\in(\prod_{v\in V}I_{v})/(I_{F^{\times}}\prod_{v\in V}I_{v})$.
Then $Q^{N}({\rm Sh})$ satisfy the following properties.
\begin{enumerate}
\item If the integrality of ${\rm Sh}$ is $m$ then ${\rm rec}(Q^{N}({\rm Sh}))=m\Theta_{S,\{\mathfrak{q}\},K}$.
\item Let ${\rm Sh}_{1}$ and ${\rm Sh}_{2}$ be Shintani data, whose integrality
are $m_{1}$ and $m_{2}$, respectively. If there exists a morphism
from ${\rm Sh}_{1}$ to ${\rm Sh}_{2}$, then
\[
\frac{m}{m_{1}}Q^{N}({\rm Sh}_{1})=\frac{m}{m_{2}}Q^{N}({\rm Sh}_{2})\ \ \ \ \ \ (m:={\rm lcm}(m_{1},m_{2})).
\]

\end{enumerate}
To illustrate the importance of Shintani data, let us sketch the proof
of Theorem \ref{thm:main_hat} for example. Assume that ${\rm ch}(\mathfrak{q})\geq[F:\mathbb{Q}]+2$.
We construct a Shintani datum ${\rm Sh}^{v_{0}}$ on $S\setminus\{v_{0}\}$
whose integrality is $1$, and define $\hat{\Theta}_{S,\mathfrak{q},S\setminus\{v_{0})}$
by $Q^{N}({\rm Sh}^{v_{0}})$. We also construct a Shintani datum
${\rm Sh}^{\diamond}$ on $S$ whose integrality is ${\rm ch}(\mathfrak{q})^{[F:\mathbb{Q}]}$.
In this paper, we introduce a technique to compute
\[
Q^{N}({\rm Sh}|_{S\setminus\{v_{0}\}})\bmod I_{F^{\times}}I_{v_{1}}\cdots I_{v_{r}}+I_{v_{0}}\cdots I_{v_{r}}
\]
for a general Shintani datum ${\rm Sh}$ on $S$. By using this technique,
we prove the following formula
\[
Q^{N}({\rm Sh}^{\diamond}|_{S\setminus\{v_{0}\}})\equiv{\rm ch}(\mathfrak{q})^{[F:\mathbb{Q}]}\hat{R}_{\mathfrak{q}}\pmod{I_{F^{\times}}I_{v_{1}}\cdots I_{v_{r}}+I_{v_{0}}\cdots I_{v_{r}}}.
\]
To relate $Q^{N}({\rm Sh}^{v_{0}})$ and $Q^{N}({\rm Sh}^{\diamond}|{}_{S\setminus\{v_{0}\}})$,
we construct a Shintani datum ${\rm Sh}^{v_{0},\diamond}$ on $S\setminus\{v_{0}\}$
whose integrality is ${\rm ch}(\mathfrak{q})^{[F:\mathbb{Q}]}$, and
morphisms ${\rm Sh}^{v_{0}}\to{\rm Sh}^{v_{0},\diamond}\leftarrow{\rm Sh}^{\diamond}|{}_{S\setminus\{v_{0}\}}$.
Hence we obtain
\[
{\rm ch}(\mathfrak{q})^{[F:\mathbb{Q}]}\hat{\Theta}_{S,\mathfrak{q},S\setminus\{v_{0})}\equiv{\rm ch}(\mathfrak{q})^{[F:\mathbb{Q}]}\hat{R}_{\mathfrak{q}}\pmod{I_{F^{\times}}I_{v_{1}}\cdots I_{v_{r}}+I_{v_{0}}\cdots I_{v_{r}}},
\]
which implies Theorem \ref{thm:main_hat}.

\subsection{The outline of the proof}

Fix $F$, $K$, $S=\{v_{0},\dots,v_{r}\}$ and $(J_{v})_{v}$. Let
$\Theta_{S,K}$ be the $S$-modified Stickelberger element. We put
$\delta_{T}:=\prod_{\mathfrak{p}\in T}(1-N(\mathfrak{p})\sigma_{\mathfrak{p}}^{-1})$.
Then we have $\Theta_{S,T,K}=\delta_{T}\Theta_{S,K}$. If $S\cap T=\emptyset$
and there exists $\mathfrak{q}\in T$ such that ${\rm ch}(\mathfrak{q})\neq2$,
then the regulator 
\[
R_{G,S,T}\in(\prod_{v\in S\setminus\{v_{0}\}}I_{G_{v}})/(I_{G_{H}}\prod_{v\in S\setminus\{v_{0}\}}I_{G_{v}})
\]
 is well-defined.

We prove Theorem \ref{Thm:main2} and Theorem \ref{thm:main_hat}
by the following steps.
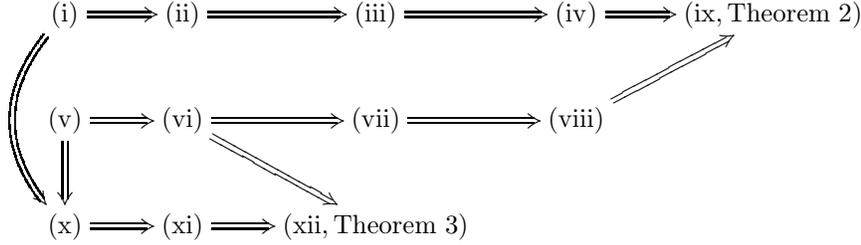
\[
\xymatrix{({\rm i})\ar@{=>}[r]\ar@/_{20pt}/@{=>}[dd] & ({\rm ii})\ar@{=>}[r] & ({\rm iii})\ar@{=>}[r] & ({\rm iv})\ar@{=>}[r] & ({\rm ix},\text{Theorem 2})\\
({\rm v})\ar@{=>}[r]\ar@{=>}[d] & ({\rm vi})\ar@{=>}[r]\ar@{=>}[rd] & ({\rm vii})\ar@{=>}[r] & ({\rm viii})\ar@{=>}[ru]\\
({\rm x})\ar@{=>}[r] & ({\rm xi})\ar@{=>}[r] & ({\rm xii},\text{Theorem 3})
}
\]
where 
\begin{enumerate}
\item \label{enu:pd_step1}For a prime ideal $\mathfrak{q}$ of $F$ such
that $\mathfrak{q}\notin S$ and ${\rm ch}(\mathfrak{q})\geq[F:\mathbb{Q}]+2$,
we construct a Shintani datum ${\rm Sh}^{v_{0}}$ on $S\setminus\{v_{0}\}$
whose integrality is $1$.
\item \label{enu:pd_step2}From (i), for all prime ideals $\mathfrak{q}$
of $F$ such that $\mathfrak{q}\notin S$ and ${\rm ch}(\mathfrak{q})\geq[F:\mathbb{Q}]+2$,
we have 
\[
\Theta_{S,\{\mathfrak{q}\},K}\in\prod_{v\in S\setminus\{v_{0}\}}I_{G_{v}}.
\]

\item \label{enu:pd_step3}From (ii), we have $\Theta_{S,T,K}\in\prod_{v\in S\setminus\{v_{0}\}}I_{G_{v}}$
since $\{1-N(\mathfrak{q})\sigma_{\mathfrak{q}}^{-1}\mid{\rm ch}(\mathfrak{q})\geq[F:\mathbb{Q}]+2,\ \mathfrak{q}\notin S\}$
generates the annihilator ideal ${\rm Ann}_{\mathbb{Z}[G]}(\mu_{K})\ni\delta_{T}$.
\item \label{enu:pd_step4}From (iii), we have $2\Theta_{S,T,K}\in2\prod_{v\in S\setminus\{v_{0}\}}I_{G_{v}}\subset I_{G_{H}}\prod_{v\in S\setminus\{v_{0}\}}I_{G_{v}}$.
\item \label{enu:pd_step5}For a prime ideal $\mathfrak{q}$ of $F$ such
that $\mathfrak{q}\notin S$ and ${\rm ch}(\mathfrak{q})\neq2$, we
construct a Shintani datum ${\rm Sh}^{\diamond}$ on $S$ whose integrality
is ${\rm ch}(\mathfrak{q})^{[F:\mathbb{Q}]}$.
\item \label{enu:pd_step6}For a prime ideal $\mathfrak{q}$ of $F$ such
that $\mathfrak{q}\notin S$ and ${\rm ch}(\mathfrak{q})\neq2$, we
prove that 
\[
Q^{N}({\rm Sh}^{\diamond}|_{S\setminus\{v_{0}\}})\equiv{\rm ch}(\mathfrak{q})^{[F:\mathbb{Q}]}\hat{R}_{\mathfrak{q}}\pmod{I_{H}I_{v_{1}}\cdots I_{v_{r}}}.
\]

\item \label{enu:pd_step7}From (vi), for a prime ideal $\mathfrak{q}$
of $F$ such that $\mathfrak{q}\notin S$ and ${\rm ch}(\mathfrak{q})\neq2$,
we have ${\rm ch}(\mathfrak{q})^{[F:\mathbb{Q}]}\Theta_{S,\{\mathfrak{q}\},K}\in\mathbb{Z}[G]$
and 
\[
{\rm ch}(\mathfrak{q})^{[F:\mathbb{Q}]}\Theta_{S,\{\mathfrak{q}\},K}\equiv{\rm ch}(\mathfrak{q})^{[F:\mathbb{Q}]}R_{G,S,\{\mathfrak{q}\}}\pmod{I_{G_{H}}\prod_{v\in S\setminus\{v_{0}\}}I_{G_{v}}}.
\]

\item \label{enu:pd_step8}Let $\mathfrak{q}$ be any element of $T$ such
that ${\rm ch}(\mathfrak{q})\neq2$. From (vii), we have
\[
{\rm ch}(\mathfrak{q})^{[F:\mathbb{Q}]}\Theta_{S,T,K}\equiv{\rm ch}(\mathfrak{q})^{[F:\mathbb{Q}]}R_{G,S,T}\pmod{I_{G_{H}}\prod_{v\in S\setminus\{v_{0}\}}I_{G_{v}}}
\]

\item \label{enu:pd_step9}From (iv) and (viii), we have
\[
\Theta_{S,T,K}\equiv R_{G,S,T}\pmod{I_{G_{H}}\prod_{v\in S\setminus\{v_{0}\}}I_{G_{v}}},
\]
which is a statement of Theorem \ref{Thm:main2}.
\item \label{enu:pd_step10}For a prime ideal $\mathfrak{q}$ of $F$ such
that ${\rm ch}(\mathfrak{q})\geq[F:\mathbb{Q}]+2$, we construct a
Shintani datum ${\rm Sh}^{v_{0},\diamond}$ on $V$, and morphisms
${\rm Sh}^{v_{0}}\to{\rm Sh}^{v_{0},\diamond}\leftarrow{\rm Sh}^{\diamond}|_{S\setminus\{v_{0}\}}$.
The integrality of ${\rm Sh}^{v_{0},\diamond}$ is ${\rm ch}(\mathfrak{q})^{[F:\mathbb{Q}]}$.
\item \label{enu:pd_step11}We put $\hat{\Theta}_{S,\mathfrak{q},S\setminus\{v_{0}\}}=Q^{N}({\rm Sh}^{v_{0}})$.
From (x), we have
\[
{\rm ch}(\mathfrak{q})^{[F:\mathbb{Q}]}\hat{\Theta}_{S,\mathfrak{q},S\setminus\{v_{0}\}}=Q^{N}({\rm Sh}^{\diamond}|{}_{S\setminus\{v_{0}\}}).
\]

\item \label{enu:pd_step12}From (vi) and (xi), we have
\[
\hat{\Theta}_{S,\mathfrak{q},S\setminus\{v_{0}\}}\equiv\hat{R}_{\mathfrak{q}}\pmod{I_{H}I_{v_{1}}\cdots I_{v_{r}}},
\]
which is a statement of Theorem \ref{thm:main_hat}.
\end{enumerate}
The hardest part of this proof is a construction of ${\rm Sh}^{v_{0}}$,
${\rm Sh}^{\diamond}$, and ${\rm Sh}^{v_{0},\diamond}$. Section
\ref{sec:Zmodule}, \ref{sec:ShintaniZeta} and \ref{sec:constSh}
are devoted to the construction of these Shintani data.

\subsection{Setting, notations and remarks}

Throughout this paper, we keep all the notations and settings in the
previous subsections except that we do not fix $T$. In addition,
the following notations are used throughout this paper.

We denote by $F_{v}$ the completion of $F$ at $v$. For $x\in F$
or $\mathbb{A}_{F}$, we denote by $x_{v}$ the image of $x$ at $F_{v}$.
For a prime ideal $\mathfrak{p}$, we write $\mathcal{O}_{\mathfrak{p}}$
for the localization of $\mathcal{O}_{F}$ at $\mathfrak{p}$, $\kappa_{\mathfrak{p}}$
for the residue field $\mathcal{O}_{F}/\mathfrak{p}$, $O_{\mathfrak{p}}$
for the maximal compact subring of $F_{\mathfrak{p}}$, and $\pi_{\mathfrak{p}}$
for a uniformizer of $O_{\mathfrak{p}}$. For a place $\mathfrak{p}\notin S$
of $F$, we denote by $\sigma_{\mathfrak{p}}\in G$ the Frobenius
element. For a set $M$ of places of $F$, we put $\mathbb{A}_{F}^{M}:=\prod'_{v\notin M}F_{v}$,
$N_{M}:=\prod_{v\in M}N_{v}$ and $N^{M}:=\prod'_{v\notin M}N_{v}$.
For a set $M$ of finite places of $F$, we write $\mathcal{O}_{M}$
for the ring of $M$-integers of $F$. We write $S_{\infty}$ for
the set of infinite places of $F$. We put $S_{f}:=S\setminus S_{\infty}$.
For a rational prime $q$, we write $S_{q}$ for the set of places
of $F$ lying above $q$. For a direct product of groups $A=\prod_{i\in\Lambda}A_{i}$
and a subset $\Lambda'\subset\Lambda$, we sometimes regard $\prod_{i\in\Lambda'}A_{i}$
as a subset of $A$, and sometimes a quotient of $A$. For a set $X$,
we write ${\bf Sub}(X)$ for the category of subsets of $X$ whose
morphisms are inclusions. For a ring $R$, we write ${\bf Mod}(R)$
for the category of $R$-modules. For a functor $\mathcal{F}:{\bf Sub}(X)\to\mathcal{C}$
and subsets $U_{1}\subset U_{2}\subset X$, we put ${\rm r}_{U_{2}}^{U_{1}}=\mathcal{F}(i_{U_{2}}^{U_{1}}):\mathcal{F}(U_{1})\to\mathcal{F}(U_{2})$
where $i_{U_{2}}^{U_{1}}$ is the unique element of ${\rm Hom}_{{\bf Sub}(X)}(U_{1},U_{2})$.
If we omit the subscript of $\otimes$, it means that the tensor product
is over $\mathbb{Z}$. For a group $A$ and $\mathbb{Z}[A]$-modules
$M_{1}$ and $M_{2}$, we regard $M_{1}\otimes M_{2}$ as a $\mathbb{Z}[A]$-module
by the standard way. For a group $A$ and an $A$-module $M$, we
put $I_{A}:=\ker(\mathbb{Z}[A]\to\mathbb{Z})$ and $I_{A}M:=\ker(M\to M\otimes_{\mathbb{Z}[A]}\mathbb{Z})$.
For a set $X$, we denote by $\bm{1}_{X}$ the characteristic function
of $X$. We denote by ${\rm lcm}(a,b)$ the least common multiple
of $a$ and $b$. For a complex $A\to B\to C$, we put $H(A\to B\to C):=\ker(B\to C)/{\rm im}(A\to B)$.
We denote by $\mathcal{S}(X)$ the set of Schwartz-Bruhat functions
from $X$ to $\mathbb{Z}$. We denote by $\mathcal{S}(X,A)$ the set
of Schwartz-Bruhat functions from $X$ to $\mathbb{Z}$ invariant
under the action of $A$. We denote by $\mathfrak{S}_{n}$ the symmetric
group of $\{1,\dots,n\}$.

For a group $E$ and an $\mathbb{Z}[E]$-module $M$, we denote by
$H_{i}(E,M)$ the $i$-th group homology. Especially, we have
\[
H_{0}(E,M)=M/I_{E}M.
\]
In this paper, we use Shapiro's Lemma 
\[
H_{i}(E_{2},M\otimes_{\mathbb{Z}[E_{1}]}\mathbb{Z}[E_{2}])=H_{i}(E_{1},M)\ \ \ \ (E_{1}\subset E_{2},\ M\ \text{is a \ensuremath{\mathbb{Z}[E_{1}]}-module})
\]
many times without mention.

\section{Stickelberger functions and elements\label{sec:pre}}

Let $T$ be a finite set of places of $F$ which is disjoint from
$S$. We define the $S$-modified and $(S,T)$-modified Stickelberger
functions to be the meromorphic $\mathbb{C}[G]$-valued functions
\begin{eqnarray*}
\Theta_{S,K}(s) & := & \prod_{\mathfrak{p}\notin S}(1-\sigma_{\mathfrak{p}}^{-1}N(\mathfrak{p})^{-s})^{-1}\\
\Theta_{S,T,K}(s) & := & \prod_{\mathfrak{p}\notin S}(1-\sigma_{\mathfrak{p}}^{-1}N(\mathfrak{p})^{-s})^{-1}\prod_{\mathfrak{p}\in T}(1-\sigma_{\mathfrak{p}}^{-1}N(\mathfrak{p})^{1-s}).
\end{eqnarray*}
We put
\begin{eqnarray*}
\Theta_{S,K} & := & \Theta_{S,K}(0)\in\mathbb{Q}[G]\\
\Theta_{S,T,K} & := & \Theta_{S,T,K}(0)\in\mathbb{Q}[G]\\
\delta_{T} & := & \prod_{\mathfrak{p}\in T}(1-N(\mathfrak{p})\sigma_{\mathfrak{p}}^{-1})\in\mathbb{Z}[G].
\end{eqnarray*}
 Then we have
\[
\Theta_{S,T,K}=\delta_{T}\Theta_{S,K}.
\]

\section{Shintani datum\label{sec:Shintani-datum}}

In this section, we introduce the notion of Shintani data and investigate
their properties. Throughout this section, we fix a quadruple $(\mathcal{R},\Upsilon,\lambda,\theta)$,
where $\mathcal{R}$ is a functor from ${\bf Sub}(S)$ to ${\bf Mod}(\mathbb{Z}[F^{\times}])$
such that the natural chain
\[
\mathcal{R}(\emptyset)\to\prod_{\substack{\#W=1\\
W\subset V
}
}\mathcal{R}(W)\to\prod_{\substack{\#W=2\\
W\subset V
}
}\mathcal{R}(W)\to\cdots\to\mathcal{R}(V)\to0
\]
is exact for all $V\subset S$, $\Upsilon$ is a $\mathbb{Z}$-module,
$\lambda$ is a homomorphism from $H_{0}(F^{\times},\mathcal{R}(\emptyset))$
to $\Upsilon$, and $\theta$ is an element of $\Upsilon$. For the
proof of the main theorems we only use the case where $(\mathcal{R},\Upsilon,\lambda,\theta)$
is as given at the start of Section \ref{sec:constSh}, but, in this
section we consider in a general setting for future research. For
$V\subset S$, we write $\mathcal{R}^{(V)}$ for the restriction of
$\mathcal{R}$ to ${\bf Sub}(V)$. We denote by ${\bf Sub}(V)\setminus\{S\}$
the full subcategory of ${\bf Sub}(V)$ consisting of $\{W\subset V\mid W\neq S\}$.
\begin{defn}
Let $V$ be a subset of $S$, $\mathcal{B}$ a functor from ${\bf Sub}(V)\setminus\{S\}$
to ${\bf Mod}(\mathbb{Z}[F^{\times}])$, $\mathcal{L}$ a natural
transform from $\mathcal{B}$ to $\mathcal{R}|_{{\bf Sub}(V)\setminus\{S\}}$,
$\vartheta$ an element of $\mathcal{B}(\emptyset)/I_{F^{\times}}\mathcal{B}(\emptyset)$,
and $m$ a positive integer. We say that the quadruple $(\mathcal{B},\mathcal{L},\vartheta,m)$
is a Shintani datum for $(\mathcal{R},\Upsilon,\lambda,\theta)$ on
$V$ (or simply a Shintani datum on $V$) if the following conditions
are satisfied.
\begin{itemize}
\item $H_{i}(F^{\times},\mathcal{B}(W))=0$ for all $i>0$ and objects $W$
of ${\bf Sub}(V)\setminus\{S\}$. 
\item $\lambda(\vartheta)=m\theta.$
\item $\bar{{\rm r}}_{\{v\}}^{\emptyset}(\vartheta)=0$ for all $v\in V$
where $\bar{{\rm r}}_{\{v\}}^{\emptyset}$ denotes a natural map from
$\mathcal{B}(\emptyset)/I_{F^{\times}}\mathcal{B}(\emptyset)$ to
$\mathcal{B}(\{v\})/I_{F^{\times}}\mathcal{B}(\{v\})$ induced by
${\rm r}_{\{v\}}^{\emptyset}$.
\end{itemize}
\end{defn}

\begin{defn}
Let $V$ be a subset of $S$. We define the category of Shintani data
on $V$ as follows. The objects are Shintani data on $V$. For Shintani
data ${\rm Sh}_{1}=(\mathcal{B}_{1},\mathcal{L}_{1},\vartheta_{1},m_{1})$
and ${\rm Sh}_{2}=(\mathcal{B}_{2},\mathcal{L}_{2},\vartheta_{2},m_{2})$
on $V$, we define the set of morphism from ${\rm Sh}_{1}$ to ${\rm Sh}_{2}$,
to be the set of natural transformation $\mathcal{J}:\mathcal{B}_{1}\to\mathcal{B}_{2}$
such that $\frac{{\rm lcm}(m_{1},m_{2})}{m_{1}}\mathcal{L}_{1}=\frac{{\rm lcm}(m_{1},m_{2})}{m_{2}}\mathcal{L}_{2}\circ\mathcal{J}$
and $\mathcal{J}(\emptyset)(\vartheta_{1})=\vartheta_{2}$.
\end{defn}

\subsection{Definition of $Q({\rm Sh})$}

Fix the free resolution of $\mathbb{Z}$ 
\[
\cdots\to\mathcal{I}_{2}\xrightarrow{\partial_{v}}\mathcal{I}_{1}\xrightarrow{\partial_{v}}\mathcal{I}_{0}\xrightarrow{\partial_{v}}\mathbb{Z}\to0
\]
in the category of $\mathbb{Z}[F^{\times}]$-modules defined by $\mathcal{I}_{k}:=\mathbb{Z}[(F^{\times})^{k+1}]$
and
\[
\partial_{v}([x_{1},\dots,x_{k}]):=\sum_{j=1}^{k}(-1)^{j-1}[x_{1},\dots,\hat{x}_{j},\dots,x_{k}].
\]
Put $\mathcal{I}_{-1}:=\mathbb{Z}$ and $\mathcal{I}_{j}:=0$ for
$j\leq-2$. Let $(\mathcal{B},\mathcal{L},\vartheta,m)$ be a Shintani
data on $V$. Recall that we put $r=\#S-1$. For $\mathcal{F}\in\{\mathcal{B},\mathcal{R}^{(V)}\}$
and $k\geq0$, we put 
\[
\mathcal{F}(k)=\begin{cases}
\bigoplus_{W\subset V,\#W=k}\mathcal{F}(W) & k\leq r\\
0 & k\geq r+1.
\end{cases}
\]
Note that $\mathcal{F}(r+1)$ is \emph{not} $\mathcal{F}(S)$ even
if $V=S$. For $k\geq1,$ define a homomorphism $\partial_{h}:\mathcal{F}(k-1)\to\mathcal{F}(k)$
by
\[
\partial_{h}((a_{W})_{W\subset V})=(b_{W})_{W\subset V}
\]
where
\[
b_{W}=\sum_{j=1}^{k}(-1)^{j-1}{\rm r}_{W}^{W\setminus\{v_{i_{j}}\}}(a_{W\setminus\{v_{i_{j}}\}})\ \ \ \ \ (W=\{v_{i_{1}},\dots,v_{i_{k}}\},i_{1}<\cdots<i_{k}).
\]
We put $\mathcal{R}^{(V)}(-1):=\ker(\mathcal{R}^{(V)}(0)\xrightarrow{\partial_{h}}\mathcal{R}^{(V)}(1))$
and $\mathcal{R}^{(V)}(j)=0$ for $j\leq-2$. We put $\mathcal{B}(j)=0$
for all $j\leq-1$. Then we get a chain complex $((\mathcal{B}(i))_{i\in\mathbb{Z}},\partial_{h})$
and $((\mathcal{R}^{(V)}(i))_{i\in\mathbb{Z}},\partial_{h})$. We
define two double complex $(\mathcal{B}_{\bullet,\bullet},\partial_{h},\partial_{v})$
and $(\mathcal{R}_{\bullet,\bullet}^{(V)},\partial_{h},\partial_{v})$
by $\mathcal{B}_{i,j}:=\mathcal{B}_{i}\otimes_{\mathbb{Z}[F^{\times}]}\mathcal{I}(j)$
and 
\[
\mathcal{R}_{i,j}^{(V)}:=\begin{cases}
\mathcal{R}^{(V)}(i)\otimes_{\mathbb{Z}[F^{\times}]}\mathcal{I}(j) & j\neq-1\\
0 & j=-1.
\end{cases}
\]
Then $\mathcal{L}$ induces a homomorphism from $\mathcal{B}_{\bullet,\bullet}$
to $\mathcal{R}_{\bullet,\bullet}^{(V)}$. We regard $\vartheta$
as an element of $\ker(H_{0}(F^{\times},\mathcal{B}(0))\to H_{0}(F^{\times},\mathcal{B}(1)))$.
For $\mathcal{F}=\mathcal{B}$ or $\mathcal{F}=\mathcal{R}^{(V)}$,
we denote by $(\mathcal{F}[\bullet],d_{\bullet})$ the total complex
of $\mathcal{F}_{\bullet,\bullet}$, i.e., we put
\[
\mathcal{F}[k]:=\bigoplus_{-i+j=k}\mathcal{F}_{i,j},
\]
\[
d_{k}:=\partial_{h}+(-1)^{k}\partial_{v}:\mathcal{F}[k]\to\mathcal{F}[k-1].
\]
The complex $(\mathcal{B}[\bullet],d_{\bullet})$ is exact since the
vertical chain $(\mathcal{B}_{i,\bullet},\partial_{v})$ is exact
for $i\in\mathbb{Z}$. If $V\neq S$ then $(\mathcal{R}^{(V)}[\bullet],d_{\bullet})$
is exact since the horizontal chain $(\mathcal{R}_{\bullet,j}^{(V)},\partial_{h})$
is exact for $j\in\mathbb{Z}$. Let us consider the following commutative
diagram\[ \xymatrix{ &  & \ker(\mathcal{B}_{0,-1}\to\mathcal{B}_{1,-1})\ar[r]^-{i} & \mathcal{B}_{0,-1}\ar[r]^{\partial_{h}}\ar[d]^{i_{1}} & \mathcal{B}_{1,-1}\ar[d]^{i_{2}}\\  & \mathcal{B}[1]\ar[r]^{d_{1}}\ar[d]^{\mathcal{L}} & \mathcal{B}[0]\ar[r]^{d_{0}}\ar[d]^{\mathcal{L}} & \mathcal{B}[-1]\ar[r]^{d_{-1}}\ar[d]^{\mathcal{L}} & \mathcal{B}[-2]\\ \mathcal{R}^{(V)}[2]\ar[r]^{d_{2}}\ar[d]^{q_{2}} & \mathcal{R}^{(V)}[1]\ar[r]^{d_{1}}\ar[d]^{q_{1}} & \mathcal{R}^{(V)}[0]\ar[r]^{d_{0}} & \mathcal{R}^{(V)}[-1]\\ \mathcal{R}_{-1,1}^{(V)}\ar[r]^{\partial_{v}} & \mathcal{R}_{-1,0}^{(V)}\ar[r]^-{q} & {\rm coker}(\mathcal{R}_{-1,1}^{(V)}\to\mathcal{R}_{-1,0}^{(V)}) } \]where
$i,i_{1},i_{2}$ are natural inclusions and $q,q_{1},q_{2}$ are natural
projections. Since $\ker(\mathcal{B}_{0,-1}\to\mathcal{B}_{1,-1})=\ker(H_{0}(F^{\times},\mathcal{B}(0))\to H_{0}(F^{\times},\mathcal{B}(1)))$,
we can regard $\vartheta$ as an element of $\ker(\mathcal{B}_{0,-1}\to\mathcal{B}_{1,-1})$.
Note that ${\rm coker}(\mathcal{R}_{-1,1}^{(V)}\to\mathcal{R}_{-1,0}^{(V)})=H_{0}(F^{\times},\mathcal{R}^{(V)}(-1))$.
\begin{defn}
\label{Def:Q}Assume that $V\neq S$. We define $Q((\mathcal{B},\mathcal{L},\vartheta,m))\in H_{0}(F^{\times},\mathcal{R}^{(V)}(-1))$
as follows. Since $d_{-1}\circ i_{1}\circ i=i_{2}\circ\partial_{h}\circ i=0$,
there exists $a\in\mathcal{B}[0]$ such that $d_{0}(a)=i_{1}\circ i(\vartheta)$.
Since $d_{0}\circ\mathcal{L}(a)=\mathcal{L}\circ i_{1}\circ i(\vartheta)=0$,
there exists $b\in\mathcal{R}^{(V)}[1]$ such that $d_{1}(b)=\mathcal{L}(a)$.
Then $q\circ q_{1}(b)$ does not depend on the choice of $a$ and
$b$. We put $Q((\mathcal{B},\mathcal{L},\vartheta,m))=q\circ q_{1}(b)$.
\end{defn}
Let ${\rm Sh}=(\mathcal{B},\mathcal{L},\vartheta,m)$ be a Shintani
datum on $V\subset S$, and $V'$ a proper subset of $V$. Then $(\mathcal{B}|_{{\bf Sub}(V')},\mathcal{L}|_{{\bf Sub}(V')},\vartheta,m)$
is a Shintani datum on $V'$. We denote this Shintani datum by ${\rm Sh}|_{V'}$.
The next lemma follows from the definition.
\begin{lem}
The map ${\rm Sh}\mapsto Q({\rm Sh})$ satisfy the following properties.
\begin{enumerate}
\item $\lambda(Q(\mathcal{B},\mathcal{L},\vartheta,m))=m\theta$
\item Let ${\rm Sh}=(\mathcal{B},\mathcal{L},\vartheta,m)$ be a Shintani
datum on $V\subsetneq S$. For all $V'\subset V$, 
\[
Q({\rm Sh}|_{V'})\equiv Q({\rm Sh})\pmod{I_{F^{\times}}\mathcal{R}^{(V')}(-1)}.
\]

\item Let ${\rm Sh}_{1}=(\mathcal{B}_{1},\mathcal{L}_{1},\vartheta_{1},m_{1})$
and ${\rm Sh}_{2}=(\mathcal{B}_{2},\mathcal{L}_{2},\vartheta_{2},m_{2})$
be a Shintani datum on $V\subsetneq S$. If there exists a morphism
from ${\rm Sh}_{1}$ to ${\rm Sh}_{2}$ then 
\[
\frac{{\rm lcm}(m_{1},m_{2})}{m_{1}}Q({\rm Sh}_{1})=\frac{{\rm lcm}(m_{1},m_{2})}{m_{2}}Q({\rm Sh}_{2}).
\]

\end{enumerate}
\end{lem}
From (ii), for $V_{2}\subset V_{1}\subsetneq S$ and a Shintani ${\rm Sh}$
on $V_{1}$, we have
\begin{equation}
Q({\rm Sh}|_{V_{2}})\equiv0\pmod{I_{F^{\times}}\mathcal{R}^{(V_{2})}(-1)+\mathcal{R}^{(V_{1})}(-1)}.\label{eq:v1v2}
\end{equation}
Unfortunately (\ref{eq:v1v2}) does not hold for $(V_{1},V_{2})=(S,S\setminus\{v_{0}\})$.
Instead, we give a way to compute 
\[
Q({\rm Sh}|_{S\setminus\{v_{0}\}})\bmod I_{F^{\times}}\mathcal{R}^{(S\setminus\{v_{0}\})}(-1)+\mathcal{R}^{(S)}(-1)
\]
 in the next section.

\subsection{Computation of $Q({\rm Sh}|_{S\setminus\{v_{0}\}})$\label{sub:CompQfull}}

Let ${\rm Sh}=(\mathcal{B},\mathcal{L},\vartheta,m)$ be a Shintani
data on $S$. Note that $H(\mathcal{R}^{(S)}[1]\to\mathcal{R}^{(S)}[0]\to\mathcal{R}^{(S)}[-1])$
is canonically isomorphic to $H_{r}(F^{\times},\mathcal{R}(S))$.
Put $V=S\setminus\{v_{0}\}$.
\begin{defn}
\label{Def:eta1}We define the map $\eta_{1}:\ker(\mathcal{B}_{0,-1}\to\mathcal{B}_{1,-1})\to H_{r}(F^{\times},\mathcal{R}(S))$
as follows. Let us consider the following commutative diagram.\[ \xymatrix{ & \ker(\mathcal{B}_{0,-1}\to\mathcal{B}_{1,-1})\ar[r]^-{i} & \mathcal{B}_{0,-1}\ar[r]^{\partial_{h}}\ar[d]^{i_{1}} & \mathcal{B}_{1,-1}\ar[d]^{i_{2}}\\ \mathcal{B}[1]\ar[r]^{d_{1}}\ar[d]^{\mathcal{L}} & \mathcal{B}[0]\ar[r]^{d_{0}}\ar[d]^{\mathcal{L}} & \mathcal{B}[-1]\ar[r]^{d_{-1}}\ar[d]^{\mathcal{L}} & \mathcal{B}[-2]\\ \mathcal{R}^{(S)}[1]\ar[r]^{d_{1}'} & \mathcal{R}^{(S)}[0]\ar[r]^{d_{0}'} & \mathcal{R}^{(S)}[-1]. } \]

Let $x\in\ker(\mathcal{B}_{0,-1}\to\mathcal{B}_{1,-1})$. Then there
exists $a\in\mathcal{B}[0]$ such that $d_{0}(a)=i_{1}\circ i(x)$.
Then $\mathcal{L}(a)\in\ker(d_{0}')$, and $(\mathcal{L}(a)\bmod{\rm im}(d_{1}'))$
does not depend on the choice of $a$. We put $\eta_{1}(x)=\mathcal{L}(a)\in\ker(d_{0}')/{\rm im}(d_{1}')\simeq H_{r-1}(F^{\times},\mathcal{R}(S))$.
\end{defn}

\begin{defn}
\label{Def:eta2}We define the map $\eta_{2}:H_{r}(F^{\times},\mathcal{R}(S))\to\mathcal{R}^{(V)}(-1)/(\mathcal{R}^{(S)}(-1)+I_{F^{\times}}\mathcal{R}^{(V)}(-1))$
as follows. Let us consider the following commutative diagram.\[
\xymatrix{ & \mathcal{R}^{(S)}[1]\ar[r]^{d_{1}'}\ar[d]^{f_{1}} & \mathcal{R}^{(S)}[0]\ar[r]^{d_{0}'}\ar[d]^{f_{0}} & \mathcal{R}^{(S)}[-1]\ar[d]^{f_{-1}}\\ \mathcal{R}^{(V)}[2]\ar[r]^{d_{2}}\ar[d]_{q_{2}} & \mathcal{R}^{(V)}[1]\ar[r]^{d_{1}}\ar[d]_{q_{1}} & \mathcal{R}^{(V)}[0]\ar[r]^{d_{0}} & \mathcal{R}^{(V)}[-1]\\ \mathcal{R}_{-1,1}^{(V)}\ar[r]^{\partial_{v}} & \mathcal{R}_{-1,0}^{(V)}\ar[r]^-{q} & \mathcal{R}_{-1,0}^{(V)}/(\mathcal{R}_{-1,0}^{(S)}+{\rm im}(\partial_{v})) }  
\]Let $y$ be an element of $H_{r}(F^{\times},\mathcal{R}(S))\simeq\ker(d_{0}')/{\rm im}(d_{1}')$.
Let $g\in\mathcal{R}^{(S)}[0]$ be a lift of $y$. Since $d_{0}\circ f_{0}(g)=f_{-1}\circ d_{0}'(g)=0$,
there exists $b\in\mathcal{R}^{(V)}[1]$ such that $d_{1}(b)=g$.
Then $q\circ q_{1}(b)$ does not depend on the choice of $g$ and
$b$ because $q\circ q_{1}\circ f_{1}=0$ and $q\circ q_{1}\circ d_{2}=q\circ\partial_{v}\circ q_{2}=0$.
We put $\eta_{2}(y)=q\circ q_{1}(b)$. \end{defn}
\begin{prop}
\label{prop:calc_of_QSH}We have
\[
Q({\rm Sh}|_{S\setminus\{v_{0}\}})\equiv\eta_{2}(\eta_{1}(\vartheta))\pmod{I_{F^{\times}}\mathcal{R}^{(V)}(-1)+\mathcal{R}^{(S)}(-1)}.
\]
\end{prop}
\begin{proof}
Let us see the diagrams in Definition \ref{Def:eta1} and \ref{Def:eta2}.
There exists $a\in\mathcal{B}[0]$ such that $d_{0}(a)=i_{1}\circ i(\vartheta)$.
Then there exists $b\in\mathcal{R}^{(V)}[1]$ such that $d_{1}(b)=f_{0}\circ\mathcal{L}(a)$.
From the definition, we have $\eta_{2}(\eta_{1}(\vartheta))=q\circ q_{1}(b)$.
Since $d_{0}(a)=i_{1}'\circ i(\vartheta)$ and $f_{0}\circ\mathcal{L}(a)=d_{1}(b)$,
we have
\[
q\circ q_{1}(b)\equiv Q({\rm Sh}|_{S\setminus\{v_{0}\}})\pmod{I_{F^{\times}}\mathcal{R}^{(V)}(-1)+\mathcal{R}^{(S)}(-1)}.
\]
Thus the proposition is proved.\end{proof}
\begin{lem}
\label{lem:eta1_is_in}Let $A$ be a subgroup of $F^{\times}$. Assume
that $H_{i}(A,\mathcal{B}(W))=0$ for all positive integer $i$ and
proper subset $W$ of $S$, and that there exists a lift $x\in\mathcal{B}(\emptyset)$
of $\vartheta$ such that ${\rm r}_{\{v\}}^{\emptyset}(x)\in I_{A}\mathcal{B}(\{v\})$
for $v\in S$. Then $\eta_{1}(\vartheta)$ is contained in ${\rm im}(H_{r}(A,\mathcal{R}(S))\to H_{r}(F^{\times},\mathcal{R}(S)))$.\end{lem}
\begin{proof}
Put $\mathcal{B}_{(A).i,j}:=\mathcal{B}_{i}\otimes_{\mathbb{Z}[A]}\mathcal{I}(j)$
for $(i,j)\in\mathbb{Z}^{2}$, $\mathcal{R}_{(A),i,j}:=\mathcal{R}(i)\otimes_{A}\mathcal{I}(j)$
for $j\neq-1$, and $\mathcal{R}_{(A),i,-1}=0$. For $\mathcal{F}=\mathcal{B}$
or $\mathcal{F}=\mathcal{R}$, put
\[
\mathcal{F}_{(A)}[k]=\bigoplus_{-i+j=k}\mathcal{F}_{(A),i,j},
\]
Consider the following commutative diagram.\[ \xymatrix{ &  &  & \mathcal{B}(\emptyset)/I_{A}\mathcal{B}(\emptyset)\ar[d]^{i'}\\  &  & \mathcal{B}_{(A)}[0]\ar[r]^{d_{0}}\ar[d]^{\mathcal{L}} & \mathcal{B}_{(A)}[-1]\ar[d]^{\mathcal{L}}\\ H_{r}(A,\mathcal{R}(S))\ar[d]^{h_{0}} & \ker(d_{0}'')\ar[l]_-{q}\ar[d]^{h_{1}}\ar[r]^{i} & \mathcal{R}_{(A)}^{(S)}[0]\ar[r]^{d_{0}''}\ar[d]^{h_{2}} & \mathcal{R}_{(A)}^{(S)}[-1]\ar[d]^{h_{3}}\\ H_{r}(F^{\times},\mathcal{R}(S)) & \ker(d_{0}')\ar[l]_-{q}\ar[r]^{i} & \mathcal{R}^{(S)}[0]\ar[r]^{d_{0}'} & \mathcal{R}^{(S)}[-1]. } \]Let
$\bar{x}\in\mathcal{B}(\emptyset)/I_{A}\mathcal{B}(\emptyset)$ be
an image of $x$. From the assumption, there exists $a'\in\mathcal{B}_{(A)}[0]$
such that $d_{0}(a')=i'(\bar{x})$. Since $d_{0}''\circ\mathcal{L}(a')=\mathcal{L}\circ d_{0}(a')=\mathcal{L}\circ i'(\bar{x})=0$,
there exists $y\in\ker(d_{0}'')$ such that $i(y)=\mathcal{L}(a')$.
Put $y':=h_{1}(y)\in\ker(d_{0}')$. Since $i(y')=h_{2}\circ\mathcal{L}(a')$,
we have $\eta_{1}(\vartheta)=q(y')=h_{0}\circ q(y)$. Thus the lemma
is proved.
\end{proof}

\section{The module $\mathcal{Z}(U,W)$\label{sec:Zmodule}}

\subsection{A certain $\mathbb{Z}$-module corresponding to a vector space}

Fix a positive integer $n$ and an $n$-dimensional vector space $V$
over $\mathbb{Q}$. For $x_{1},\dots,x_{k}\in V$, we denote by $C(x_{1},\dots,x_{k})$
the open cone generated by $x_{1},\dots,x_{k}$ in $V$ or $V\otimes\mathbb{R}$.
\begin{defn}
We say that a subset $U$ of $V\setminus\{0\}$ is fat if $U$ cannot
be covered by any finite union of proper subspaces of $V$. 
\end{defn}
Fix a fat subset $U$ of $V\setminus\{0\}$. For $k\geq1$, we denote
by $X_{k}(U)$ the $\mathbb{Z}$-module generated by the formal symbol
$[x_{1},\dots,x_{k}]$ where $x_{1},\dots,x_{k}$ are linearly independent
vectors in $U$. We put $X(U):=\bigoplus_{k=1}^{n}X_{k}(U)$ and $X_{{\rm low}}(U):=\bigoplus_{k=1}^{n-1}X_{k}(U)\subset X(U).$
We define a homomorphism $\mathcal{L}_{\infty}:X(U)\to{\rm Map}(V,\mathbb{Z})$
by
\[
\mathcal{L}_{\infty}([x_{1},\dots,x_{k}]):=\bm{1}_{C(x_{1},\dots,x_{k})}.
\]
We put 
\[
\mathcal{K}'(U):={\rm im}(X(U)\xrightarrow{\mathcal{L}_{\infty}}{\rm Map}(V,\mathbb{Z}))
\]
and
\[
\mathcal{K}(U):=\mathcal{K}'(U)/\left(\mathcal{K}'(U)\cap\bm{1}_{V\setminus\{0\}}\mathbb{Z}\right).
\]
In this section, we define a subset $Y(U)$ of $X(U)$, prove that
$Y(U)\subset\ker(X(U)\xrightarrow{\mathcal{L}_{\infty}}\mathcal{K}(U))$
(Proposition \ref{prop:LYUvanish}), and construct a certain exact
sequence (Proposition \ref{prop:exactXZC}).

For $m\geq0$, we denote by $C_{m}(U)$ the $\mathbb{Z}$-module generated
by the formal symbol
\[
(x_{1},\dots,x_{m})
\]
where $x_{1},\dots,x_{m}\in U$ are in general position. We define
$\partial_{m}:C_{m}(U)\to C_{m-1}(U)$ by
\[
\partial_{m}((x_{1},\dots,x_{m}))=\sum_{j=1}^{m}(-1)^{j-1}(x_{1},\dots,\widehat{x_{j}},\dots,x_{m}).
\]
Since $U$ is fat, the following sequence is exact.
\[
\cdots\to C_{3}(U)\to C_{2}(U)\to C_{1}(U)\to C_{0}(U)\to0.
\]
Let us fix a map $r:V^{n}\to\{0,1,-1\}$ such that
\begin{itemize}
\item $r(x_{1},\dots,x_{n})\neq0$ if and only if $x_{1},\dots,x_{n}$ are
linearly independent,
\item $r(fx_{1},\dots,fx_{n})={\rm sgn}(\det(f))r(x_{1},\dots,x_{n})$ for
all automorphisms $f$ of $V$. 
\end{itemize}
We call such a map an orientation of $V$. We define the homomorphism
$\psi:C_{n+1}(U)\to X(U)$ by
\[
\psi((x_{1},\dots,x_{n+1}))=\sum_{u\in\{\pm1\}}\sum_{k=1}^{n}\sum_{i_{1},\dots,i_{k}}u[x_{i_{1}},\dots,x_{i_{k}}]
\]
where $1\leq i_{1}<\cdots<i_{k}\leq n+1$ runs all tuples such that
\[
(-1)^{j-1}r(x_{1},\dots,\hat{x}_{j},\dots,x_{n+1})=u
\]
for all $j\in\{1,\dots,n+1\}\setminus\{i_{1},\dots,i_{k}\}$. We put
$Y(U)={\rm Image}(\psi)\subset X(U)$ and $Z(U)=X(U)/Y(U)$. Note
that $Y(U)$ does not depend on the choice of $r$. We say that $Q\in V\otimes\mathbb{R}$
is an irrational vector if $Q$ is not contained in any proper subspace
of $V\otimes\mathbb{R}$ spanned by vectors in $V$. For an irrational
vector $Q$, define the homomorphism $\varphi^{Q}:C_{n}(U)\to X(U)$
by
\[
\varphi^{Q}((x_{1},\dots,x_{n}))=r(x_{1},\dots,x_{n})\sum_{k=1}^{n}\sum_{i_{1},\dots,i_{k}}[x_{i_{1}},\dots,x_{i_{k}}]
\]
where $1\leq i_{1}<\cdots<i_{k}\leq n$ runs all tuple such that
\[
\frac{r_{x_{j}\to Q}(x_{1},\dots,x_{n})}{r(x_{1},\dots,x_{n})}>0
\]
for all $j\in\{1,\dots,n\}\setminus\{i_{1},\dots,i_{k}\}$, where
$r_{x_{j}\to Q}(x_{1},\dots,x_{n})$ denotes $r(x_{1},\dots x_{j-1},Q,x_{j+1},\dots,x_{n})$.
\begin{lem}
\label{lem:Hill}Let $x_{1},\dots,x_{n+1},y\in V\otimes\mathbb{R}$
be vectors in general position. Then we have
\begin{multline*}
\sum_{j=1}^{n+1}(-1)^{j-1}r(x_{1},\dots,\widehat{x_{j}},\dots,x_{n+1})\bm{1}_{C(x_{1},\dots,\widehat{x_{j}},\dots,x_{n+1})}(y)\\
=\begin{cases}
1 & (-1)^{j-1}r(x_{1},\dots,\hat{x}_{j},\dots,x_{n+1})=1\ \text{for all}\ j=1,\dots,n+1\\
-1 & (-1)^{j-1}r(x_{1},\dots,\hat{x}_{j},\dots,x_{n+1})=-1\ \text{for all}\ j=1,\dots,n+1\\
0 & {\rm otherwise}.
\end{cases}
\end{multline*}
\end{lem}
\begin{proof}
It was proved in \cite[Proposition 2]{MR2392823}. \end{proof}
\begin{lem}
\label{lemPsiPhi}Let $Q\in V\otimes\mathbb{R}$ be an irrational
vector. We have $\psi(a)=\varphi^{Q}(\partial a)$ for all $a\in C_{n+1}(U)$.\end{lem}
\begin{proof}
Put $a=(x_{1},\dots,x_{n+1})$. We put $I=\{(i_{1},\dots,i_{k})\mid1\leq i_{1}<\cdots<i_{k}\leq n+1\}$.
From the definition of $\varphi^{Q}$ and $\partial a$, we have
\[
\varphi^{Q}(\partial a)=\sum_{(i_{1},\dots,i_{k})\in I}w(i_{1},\dots,i_{k})[i_{1},\dots,i_{k}]
\]
where $w(i_{1},\dots,i_{k})$ is an integer defined by 
\[
w(i_{1},\dots,i_{k}):=\sum_{m}(-1)^{m-1}r(x_{1},\dots,\widehat{x_{m}},\dots,x_{n+1})
\]
where $m$ runs all integers between $1$ and $n+1$ such that $m\notin\{i_{1},\dots,i_{k}\}$
and
\begin{equation}
\frac{r_{x_{j}\to Q}(x_{1},\dots,\widehat{x_{m}},\dots,x_{n+1})}{r(x_{1},\dots,\widehat{x_{m}},\dots,x_{n+1})}>0\label{eq:condPositive}
\end{equation}
for all $j\in\{1,\dots,n+1\}\setminus\{i_{1},\dots,i_{k},m\}$. Fix
$(i_{1},\dots,i_{k})\in I$. Put $V':=V/(x_{i_{1}}\mathbb{Q}+\cdots+x_{i_{k}}\mathbb{Q})$.
For $x\in V$, we denote by $\bar{x}$ the image of $x$ in $V'$.
Then the condition (\ref{eq:condPositive}) is equivalent to
\begin{equation}
\bar{Q}\in C((\bar{x}_{j})_{j\neq i_{1},\dots,i_{k},m}).\label{eq:condPositive2}
\end{equation}
Put 
\[
\{j_{1}\leq\cdots\leq j_{n+1-k}\}:=\{1,\dots,n+1\}\setminus\{i_{1},\dots,i_{k}\},
\]
$m=j_{c}$ and $y_{\ell}:=x_{j_{\ell}}$. Then (\ref{eq:condPositive2})
is equivalent to
\[
\bar{Q}\in C(y_{1},\dots,\widehat{y_{c}},\dots,y_{n+1-k}).
\]
Thus we have
\[
w(i_{1},\dots,i_{k})=\sum_{c=1}^{n+1-k}(-1)^{j_{c}-1}r(x_{1},\dots,\widehat{x_{j_{c}}},\dots,x_{n+1})\bm{1}_{C(y_{1},\dots,\widehat{y_{c}},\dots,y_{n+1-k})}(\bar{Q}).
\]
Take an orientation $r'$ of $V'$ such that
\[
(-1)^{c-1}r'(y_{1},\dots,\widehat{y_{c}},\dots,y_{n+1-k})=(-1)^{j_{c}-1}r(x_{1},x_{2},\dots,\widehat{x_{j_{c}}},\dots,x_{n},x_{n+1})\ \ \ (1\leq c\leq n+1-k).
\]
Then we have
\begin{align*}
w(i_{1},\dots,i_{k}) & =\sum_{c=1}^{n+1-k}(-1)^{c-1}r'(y_{1},\dots,\widehat{y_{c}},\dots,y_{n+1-k})\bm{1}_{C(y_{1},\dots,\hat{y}_{c},\dots,y_{n+1-k})}(\bar{Q}).
\end{align*}
Thus from Lemma \ref{lem:Hill}, we have
\begin{eqnarray*}
w(i_{1},\dots,i_{k}) & = & \begin{cases}
1 & (-1)^{c-1}r'(y_{1},\dots,\widehat{y_{c}},\dots,y_{n+1-k})=1\ \text{for all}\ c\\
-1 & (-1)^{c-1}r'(y_{1},\dots,\widehat{y_{c}},\dots,y_{n+1-k})=-1\ \text{for all}\ c\\
0 & {\rm otherwise}
\end{cases}\\
 & = & \begin{cases}
1 & (-1)^{j-1}r(x_{1},\dots,\hat{x}_{j},\dots,x_{n+1})=1\ \text{for all}\ j\notin\{i_{1},\dots,i_{k}\}\\
-1 & (-1)^{j-1}r(x_{1},\dots,\hat{x}_{j},\dots,x_{n+1})=-1\ \text{for all}\ j\notin\{i_{1},\dots,i_{k}\}\\
0 & {\rm otherwise}.
\end{cases}
\end{eqnarray*}
Thus the claim is proved.\end{proof}
\begin{prop}
\label{prop:LYUvanish}We have $Y(U)\subset\ker(X(U)\xrightarrow{\mathcal{L}_{\infty}}\mathcal{K}(U))$.\end{prop}
\begin{proof}
From Lemma \ref{lemPsiPhi}, it is enough to prove that
\[
\mathcal{L}_{\infty}(\varphi^{Q}(\partial a))\in\bm{1}_{V\setminus\{0\}}\mathbb{Z}
\]
for all $a\in C_{n+1}(U)$. For $(x_{1},\dots,x_{n})\in C_{n}(U)$,
we have
\[
\mathcal{L}_{\infty}(\varphi^{Q}((x_{1},\dots,x_{n}))=r(x_{1},\dots,x_{n})\cdot f_{Q,(x_{1},\dots,x_{n})}
\]
where $f_{Q,(x_{1},\dots,x_{n})}:V\to\mathbb{Z}$ is a map defined
by
\[
f_{Q,(x_{1},\dots,x_{n})}(z):=\begin{cases}
\lim_{\epsilon\to+0}\bm{1}_{C(x_{1},\dots,x_{n})}(z+\epsilon Q) & z\neq0\\
0 & z=0.
\end{cases}
\]
Thus we have
\begin{align*}
\mathcal{L}_{\infty}(\varphi^{Q}(\partial(x_{1},\dots,x_{n+1})))(z) & =\sum_{j=1}^{n+1}(-1)^{j-1}r(x_{1},\dots,\widehat{x_{j}},\dots,x_{n+1})\bm{1}_{C(x_{1},\dots,\widehat{x_{j}},\dots,x_{n+1})}(z+\epsilon Q)\\
 & =(\text{constant function for }z\in V\setminus\{0\})
\end{align*}
from Lemma \ref{lem:Hill}. Thus the claim is proved.
\end{proof}

\begin{defn}
For $a\in\ker(C_{n}(U)\to C_{n-1}(U))$, we define $\varphi(a)\in X(U)$
by $\psi(b)$ where $b$ is any element of $C_{n+1}(U)$ such that
$\partial b=a$. This definition does not depend on the choice of
$b$.\end{defn}
\begin{lem}
The natural map $X_{{\rm low}}(U)\to Z(U)$ is injective.\end{lem}
\begin{proof}
The claim is equivalent to $X_{{\rm low}}(U)\cap Y(U)=\{0\}$. Let
$a\in C_{n+1}(U)$ such that $\psi(a)\in X_{{\rm low}}(U)$. Since
$\psi(a)\in X_{{\rm low}}(U)$, we have $\partial_{n+1}(a)=0$. Let
$Q$ be any irrational vector. Then we have
\[
\psi(a)=\varphi^{Q}(\partial_{n+1}(a))=\varphi^{Q}(0)=0.
\]
Thus the claim is proved. 
\end{proof}
Define a homomorphism $\phi:X(U)\to C_{n}(U)$ by
\[
\phi([x_{1},\dots,x_{m}])=\begin{cases}
r(x_{1},\dots,x_{n})\cdot(x_{1},\dots,x_{n}) & m=n\\
0 & m<n.
\end{cases}
\]
Put $\bar{C}_{n}(U)=C_{n}(U)/\ker(\partial_{n})$. Then $\varphi$
induces an isomorphism from $Z(U)/X_{{\rm low}}(U)$ to $\bar{C}_{n}(U)$.
Thus we get the following proposition.
\begin{prop}
\label{prop:exactXZC}The following sequence is exact:
\[
0\to X_{{\rm low}}(U)\to Z(U)\xrightarrow{\phi}\bar{C}_{n}(U)\to0.
\]

\end{prop}

\subsection{The group action on $V_{{\rm ad}}$}

Let $E$ be a group which acts linearly on $V$ and freely on $V\setminus\{0\}$.
Let $U\subset V$ be a fat subset closed under the action of $E$.
For $\epsilon\in E$, put ${\rm sgn}(\epsilon):={\rm sgn}(\det(\epsilon:V\to V)).$
The $\mathbb{Z}$-module $C_{m}(U)$ has a two kind of structure of
$\mathbb{Z}[E]$-module. The one is defined by
\[
[\epsilon](x_{1},\dots,x_{m})=(\epsilon x_{1},\dots,\epsilon x_{m})\ \ \ (\epsilon\in E,\ x_{1},\dots,x_{m}\in U),
\]
and the other is defined by
\[
[\epsilon](x_{1},\dots,x_{m})={\rm sgn}(\epsilon)(\epsilon x_{1},\dots,\epsilon x_{m})\ \ \ (\epsilon\in E,\ x_{1},\dots,x_{m}\in U).
\]
We write $C_{m}^{+}(U)$ for the first $\mathbb{Z}[E]$-module and
$C_{m}^{-}(U)$ for the second one. Let $\mathbb{Z}^{-}$ be the $\mathbb{Z}[E]$-module
whose underlying $\mathbb{Z}$-module is $\mathbb{Z}$, and define
the action of $E$ to $\mathbb{Z}^{-}$ by
\[
\epsilon n={\rm sgn}(\epsilon)n\ \ \ \ (\epsilon\in E,\ n\in\mathbb{Z}^{-}).
\]
Note that we have $C_{0}^{+}(V_{{\rm ad}})\simeq\mathbb{Z}$ and $C_{0}^{-}(V_{{\rm ad}})\simeq\mathbb{Z}^{-}.$
The purpose of this subsection is to define the homomorphism 
\[
\Omega:H_{n-1}(E,\mathbb{Z}^{-})\to H_{0}(E,Z(U))=Z(U)/I_{E}Z(U).
\]

\begin{defn}
For $\bm{x}=(x_{1},\dots,x_{k})\in U^{k}$, we denote by $A_{m}(\bm{x})$
(resp. $A_{m}'(\bm{x})$) the subgroup of $C_{m}(U)$ spanned by
\[
(y_{1},\dots,y_{m})\in C_{m}(U)
\]
where $y_{1},\dots,y_{n}$ are elements of $U$ such that $(x_{1},\dots,x_{k})$
is (resp. is not) a subsequence of $(y_{1},\dots,y_{m})$.
\end{defn}
From the definition, we have
\[
C_{m}(U)=A_{m}(\bm{x})\oplus A_{m}'(\bm{x}).
\]

\begin{defn}
For $\bm{x}=(x_{1},\dots,x_{k})\in U^{k}$ and $z\in X(U)$, we denote
by ${\rm coeff}(\bm{x},z)\in\mathbb{Z}$ the coefficient of $[x_{1},\dots,x_{k}]$
in $z$.
\end{defn}

\begin{defn}
For $\bm{x}\in U^{k}$ and $a\in C_{n}(U)$ such that $\partial a\in A_{n-1}'(\bm{x})$,
define ${\rm coeff}^{\#}(\bm{x},a)\in\mathbb{Z}$ by
\[
{\rm coeff}^{\#}(\bm{x},a):={\rm coeff}(\bm{x},\varphi(a-a'))
\]
where $a'$ is an element of $A_{n}'(\bm{x})$ such that $\partial(a-a')=0$.
Such an element $a'$ always exists and ${\rm coeff}^{\#}(\bm{x},a)$
does not depend on the choice of $a'$.
\end{defn}
In other words, ${\rm coeff}^{\#}(\bm{x},-)$ is a unique homomorphism
from $\partial^{-1}(A_{n-1}'(\bm{x}))$ to $\mathbb{Z}$ such that
\begin{align*}
{\rm coeff}^{\#}(\bm{x},a) & ={\rm coeff}(\bm{x},\varphi(a))\ \ \ \ (a\in\ker\partial)\\
{\rm coeff}^{\#}(\bm{x},a) & =0\ \ \ \ (a\in A_{n}'(\bm{x})).
\end{align*}

\begin{defn}
Let $\bm{x}\in U^{k}$ and $a\in\partial_{n}^{-1}(I_{E}C_{n-1}(U))$.
We define ${\rm coeff}_{E}^{\#}(\bm{x},a)\in\mathbb{Z}$ by
\[
{\rm coeff}_{E}^{\#}(\bm{x},a,E):=\lim_{\substack{Z\to E\\
\#Z<\infty
}
}{\rm coeff}^{\#}(\bm{x},\sum_{\epsilon\in Z}\epsilon a).
\]
Here the right hand side means ${\rm coeff}^{\#}(\bm{x},\sum_{\epsilon\in Z_{0}}\epsilon a)$
where $Z_{0}$ is an enough large finite subset of $E$ such that
$\sum_{\epsilon\in Z}\epsilon a\in\ker\partial$ and
\[
{\rm coeff}^{\#}(\bm{x},\sum_{\epsilon\in Z_{0}}\epsilon a)={\rm coeff}^{\#}(\bm{x},\sum_{\epsilon\in Z}\epsilon a)
\]
for all finite subset $Z$ of $E$ which contains $Z_{0}$.
\end{defn}

\begin{defn}
For $a\in\partial_{n}^{-1}(I_{E}C_{n-1}(U))/I_{E}C_{n}(U)$, we define
$\varphi_{E}(a)\in H_{0}(E,X(U))$ by
\[
\sum_{k=1}^{n}\sum_{(x_{1},\dots,x_{k})\in U^{k}/E}{\rm coeff}_{E}^{\#}((x_{1},\dots,x_{k}),a)\cdot[x_{1},\dots,x_{k}].
\]
\end{defn}
\begin{lem}
For all $a\in\ker\partial_{n}$, we have
\[
\varphi_{E}(a)\equiv\varphi(a)\pmod{I_{E}X(U)}.
\]
\end{lem}
\begin{proof}
For $\bm{x}=(x_{1},\dots,x_{k})\in U^{k}/E$ and $z\in H_{0}(E,X(U))$,
we put 
\[
{\rm coeff}_{E}(\bm{x},z)=\lim_{Z\to X}{\rm coeff}(\bm{x},\sum_{\epsilon\in Z}\epsilon z).
\]
Then it is enough to prove that
\[
{\rm coeff}_{E}(\bm{x},\varphi_{E}(a))={\rm coeff}_{E}(\bm{x},\varphi(a)).
\]
The left hand side is equal to
\begin{align*}
{\rm coeff}_{E}^{\#}((x_{1},\dots,x_{k}),a) & =\lim_{\substack{Z\to E\\
\#Z<\infty
}
}{\rm coeff}^{\#}(\bm{x},\sum_{\epsilon\in Z}\epsilon a)\\
 & =\lim_{\substack{Z\to E\\
\#Z<\infty
}
}{\rm coeff}(\bm{x},\sum_{\epsilon\in Z}\epsilon\varphi(a))\\
 & ={\rm coeff}_{E}(\bm{x},\varphi(a)).
\end{align*}
Thus the lemma is proved.
\end{proof}
Since
\[
\cdots\to C_{2}^{-}(U)\to C_{1}^{-}(U)\to C_{0}^{-}(U)\simeq\mathbb{Z}^{-}\to0
\]
is a free resolution of $\mathbb{Z}^{-}$ in the category of $\mathbb{Z}[E]$-module,
there exists a natural isomorphism
\[
H_{n-1}(E,\mathbb{Z}^{-})\simeq\partial_{n}^{-1}(I_{E}C_{n-1}^{-}(U))/(\ker\partial_{n}+I_{E}C_{n}^{-}(U)).
\]
By composing this isomorphism and $\varphi_{E}$, we get the homomorphism
\[
\Omega_{E,U}:H_{n-1}(E,\mathbb{Z}^{-})\to H_{0}(E,Z(U)).
\]
Note that the diagram
\[
\xymatrix{H_{n-1}(E_{1},\mathbb{Z}^{-})\ar[r]^{\Omega_{E_{1},U_{1}}}\ar[d] & H_{0}(E_{1},Z(U_{1}))\ar[d]\\
H_{n-1}(E_{2},\mathbb{Z}^{-})\ar[r]^{\Omega_{E_{2},U_{2}}} & H_{0}(E_{2},Z(U_{2}))
}
\]
commutes for $E_{1}\subset E_{2}$ and $U_{1}\subset U_{2}$ where
two vertical arrows in the diagram are natural maps.
\begin{lem}
\label{lem:QE}Let $Q\in F\otimes\mathbb{R}$ be an irrational vector
such that $\epsilon Q\in\mathbb{R}_{>0}Q$ for all $\epsilon\in E$.
For $a\in\partial_{n}^{-1}(I_{E}C_{n-1}(U))$, we have $\varphi_{E}(a)\equiv\varphi^{Q}(a)\pmod{I_{E}Z(U)}$.\end{lem}
\begin{proof}
Fix $\bm{x}=(x_{1},\dots,x_{k})\in U^{k}$ and $a\in\partial_{n}^{-1}(I_{E}C_{n-1}(U))$.
It is enough to prove that
\begin{equation}
{\rm coeff}_{E}(\bm{x},\varphi_{E}(a))={\rm coeff}_{E}(\bm{x},\varphi^{Q}(a)).\label{eq:eEQ}
\end{equation}
Since $a\in\partial_{n}^{-1}(I_{E}C_{n-1}(U))$, if $Z$ is large
enough, there exists $b_{Z}\in A_{n}'(\bm{x})$ such that
\[
\sum_{\epsilon\in Z}\epsilon a-b_{Z}\in\ker(\partial_{n}).
\]
The left hand side of (\ref{eq:eEQ}) is equal to
\begin{align*}
{\rm coeff}_{E}^{\#}(x,a) & =\lim_{\substack{Z\to E\\
\#Z<\infty
}
}{\rm coeff}^{\#}(\bm{x},\sum_{\epsilon\in Z}\epsilon a)\\
 & =\lim_{\substack{Z\to E\\
\#Z<\infty
}
}{\rm coeff}^{\#}(\bm{x},-b_{Z}+\sum_{\epsilon\in Z}\epsilon a)\\
 & =\lim_{\substack{Z\to E\\
\#Z<\infty
}
}{\rm coeff}(\bm{x},\varphi(-b_{Z}+\sum_{\epsilon\in Z}\epsilon a)).
\end{align*}
Note that we have $\epsilon\varphi^{Q}(a)=\varphi^{Q}(\epsilon a)$
for $\epsilon\in E$ since $\epsilon Q\in\mathbb{R}_{>0}Q$. The right
hand side of (\ref{eq:eEQ}) is equal to 
\begin{align*}
\lim_{\substack{Z\to E\\
\#Z<\infty
}
}{\rm coeff}(x,\sum_{\epsilon\in Z}\epsilon\varphi^{Q}(a)) & =\lim_{\substack{Z\to E\\
\#Z<\infty
}
}{\rm coeff}(x,\sum_{\epsilon\in Z}\varphi^{Q}(\epsilon a))\\
 & =\lim_{\substack{Z\to E\\
\#Z<\infty
}
}{\rm coeff}_{E}(x,\varphi^{Q}(-b_{Z}+\sum_{\epsilon\in Z}\epsilon a))\\
 & =\lim_{\substack{Z\to E\\
\#Z<\infty
}
}{\rm coeff}_{E}(x,\varphi(-b_{Z}+\sum_{\epsilon\in Z}\epsilon a)).
\end{align*}
Hence (\ref{eq:eEQ}) is proved.
\end{proof}
The next lemma follows from a simple calculation.
\begin{lem}
\label{lem:s1_V_ptimes}Let $p$ be a $\mathbb{Q}[E]$-automorphism
of $V$. Let $U_{1}$ and $U_{2}$ be fat subsets of $F^{\times}$
such that $pU_{1}\subset U_{2}$. Let 
\[
p':H_{0}(E,Z(U_{1}))\to H_{0}(E,Z(U_{2}))
\]
be a homomorphism induced by $p$. For $x\in H_{n-1}(E,\mathbb{Z}^{-})$,
we have
\[
p'\circ\Omega_{E,U_{1}}(x)={\rm sgn}(\det(p))\times\Omega_{E,U_{2}}(x).
\]

\end{lem}

\subsection{The module $\mathcal{Z}(U,W)$}

For $W\subset S_{\infty}$, put $X_{W}=\prod_{v\in S_{\infty}\setminus W}F_{v}^{\times}/\mathbb{R}_{>0}$.
We denote by $i_{W}$ the natural diagonal map from $F^{\times}$
to $X_{W}$. For groups $E_{1}\subset E_{2}$ and an $\mathbb{Z}[E_{2}]$-module
$M$, let $\alpha_{E_{1},E_{2}}:H_{0}(E_{1},M)\to H_{0}(E_{2},M)$
be a natural homomorphism.

Let $U$ be a subset of $F^{\times}$ such that $U\cap i_{\emptyset}^{-1}(g)$
is fat for all $g\in X_{\emptyset}$. For example, this condition
is satisfied if $U$ is dense in $F\otimes\mathbb{R}$. For $W\subset S_{\infty}$
and $g\in X_{W}$, we put $U_{g}:=U\cap i_{W}^{-1}(g)$. We define
the functors $\mathcal{Z}(U,-)$ from ${\bf Sub}(S_{\infty})$ to
${\bf Mod}(\mathbb{Z})$ by 
\[
\mathcal{Z}(U,W)=\bigoplus_{g\in X_{W}}Z(U_{g}).
\]

\begin{prop}
\label{prop:Hi_EZUW_vanish}Let $E$ be a subgroup of $\mathcal{O}_{F}^{\times}$.
Assume that $U$ is closed under the action of $E$. For $W\subset S_{\infty}$
and $i>0$, we have 
\[
H_{i}(E,\mathcal{Z}(U,W))=0
\]
if $W\subsetneq S_{\infty}$ or $-1\notin E$. \end{prop}
\begin{proof}
Fix $i>0$. Put $E'=E\cap i_{W}^{-1}(1)$. Then $X_{W}$ can be written
as
\[
X_{W}=\bigsqcup_{\epsilon\in E/E'}\epsilon C
\]
where $C$ is a subset of $X_{W}$. Put 
\[
Z_{C}=\bigoplus_{g\in C}Z(U_{g}).
\]
Then we have
\[
\mathcal{Z}(U,W)=Z_{C}\otimes_{\mathbb{Z}[E']}\mathbb{Z}[E].
\]
Thus we have
\begin{align*}
H_{i}(E,\mathcal{Z}(U,W)) & =H_{i}(E',Z_{C})\\
 & =\bigoplus_{g\in C}H_{i}(E',Z(U_{g})).
\end{align*}
Fix $g\in C$. From Proposition \ref{prop:exactXZC}, there exists
an exact sequence of $\mathbb{Z}[E']$-module:
\[
0\to X_{{\rm low}}(U_{g})\to Z(U_{g})\xrightarrow{\phi_{r}}\bar{C}_{n}^{-}(U_{g})\to0.
\]
Since $X_{{\rm low}}(U_{g})$ is a free $E'$-module, we have
\[
H_{i}(E',X_{{\rm low}}(U_{g}))=0\ \ \ (i>0).
\]
Thus it is enough to prove that $H_{i}(E',\bar{C}_{n}^{-}(U_{g}))=0$
for $i>0$. Put $E'':=\ker({\rm sgn}:E'\to\{\pm1\})$. Note that $E'$
and $E''$ are free abelian groups of rank $n-1$ from the condition
$(W\neq S_{\infty})\lor(-1\notin E)$. In the case $E''=E'$, we have
\[
H_{i}(E',\bar{C}_{n}^{-}(U_{g}))=H_{i+n-1}(E',\mathbb{Z})=0.
\]
In the case $E''\neq E'$, there exists an exact sequence
\[
0\to\bar{C}_{n}^{-}(U_{g})\to\bar{C}_{n}^{+}(U_{g})\otimes_{\mathbb{Z}[E'']}\mathbb{Z}[E']\to\bar{C}_{n}^{+}(U_{g})\to0.
\]
For $i>0$, we have 
\[
H_{i}(E',\bar{C}_{n}^{+}(U_{g}))=H_{i+n-1}(E',\mathbb{Z})=0
\]
and
\begin{align*}
H_{i}(E',\bar{C}_{n}^{+}(U_{g})\otimes_{\mathbb{Z}[E'']}\mathbb{Z}[E']) & =H_{i}(E'',\bar{C}_{n}^{+}(U_{g}))\\
 & =H_{i+n-1}(E'',\mathbb{Z})\\
 & =0.
\end{align*}
Thus the claim is proved.
\end{proof}
Let $\rho_{1},\dots,\rho_{n}$ be the distinct real embeddings of
$F$. Define an orientation $r:F^{n}\to\{0,1,-1\}$ by $r(x_{1},\dots,x_{n}):={\rm sgn}(\det(\rho_{i}(x_{j}))_{i,j=1}^{n})$.
From the definition, $\mathcal{L}_{\infty}$ induces a map
\[
\mathcal{Z}(U,W)\to\bigoplus_{g\in X_{W}}\mathcal{K}(U_{g}).
\]
If $W\neq S_{\infty}$ then $\bigoplus_{g\in X_{W}}\mathcal{K}(U_{g})$
can be viewed as a subgroup of ${\rm Map}(F^{\times},\mathbb{Z})$
by the obvious way. Therefore we obtain a natural map from $\mathcal{Z}(U,W)$
to ${\rm Map}(F^{\times},\mathbb{Z})$ for $W\neq S_{\infty}$. By
abuse of notation, we denote this map by $\mathcal{L}_{\infty}$.
We denote by $\mathcal{O}_{F,+}^{\times}$ the group of totally positive
units of $F$. Let $E$ be a finite index subgroup of $\mathcal{O}_{F,+}^{\times}$.
Then $H_{n-1}(E,\mathbb{Z})$ is isomorphic to $\mathbb{Z}$. The
canonical generator $\eta_{E}$ of $H_{n-1}(E,\mathbb{Z})$ is defined
by the following Lemma. 
\begin{lem}
\label{lem:tot_fundamental_domain}There exists a generator $\eta_{E}$
of $H_{n-1}(E,\mathbb{Z})$ such that
\[
\sum_{\epsilon\in E}\epsilon\mathcal{L}_{\infty}(a)={\rm sgn}(g)\bm{1}_{i_{\emptyset}^{-1}(g)}
\]
for any element $g$ of $X_{\emptyset}$ and lift $a\in Z(U_{g})$
of $\Omega_{E,U_{g}}(\eta_{E})\in H_{0}(E,Z(U_{g}))$.\end{lem}
\begin{proof}
We put $D_{k,g}:=\mathbb{Z}[(i_{\emptyset}^{-1}(g))^{k}]$. Then $C_{k}(U_{g})$
can be naturally embedded in $D_{k,g}$. Take a $\mathbb{Z}$-basis
$\left\langle \epsilon_{1},\dots,\epsilon_{n-1}\right\rangle $ of
$E$ such that $\det(\log(\rho_{i}(\epsilon_{j})))_{i,j=1}^{n-1}>0$.
For $\sigma\in\mathfrak{S}_{n-1}$ and $0\leq j\leq n-1$, we put
\[
e_{j,\sigma}:=\prod_{k=1}^{j}\epsilon_{\sigma(k)}.
\]
For $x\in i_{\emptyset}^{-1}(g)$, we put 
\[
F(x):=\sum_{\sigma\in\mathfrak{S}_{n-1}}(xe_{0,\sigma},\dots,xe_{n-1,\sigma})\in D_{n,g}.
\]
Then we have $F(x)\in\ker(D_{n,g}\to D_{n-1,g})$. Let $\eta_{E}\in H_{n-1}(E,\mathbb{Z})$
be the image of $F(x)$ under the composite map
\[
\ker(D_{n,g}\to D_{n-1,g})\to\ker(D_{n,g}\to D_{n-1,g})/{\rm im}(D_{n+1,g}\to D_{n,g})\simeq H_{n-1}(E,\mathbb{Z}).
\]
Then $\eta_{E}$ does not depend on the choice of $g$ and $x$. Fix
an irrational vector $Q\in F\otimes\mathbb{R}$ such that $\rho_{j}(Q)=0$
for $1\leq j\leq n-1$. We define $\bar{\varphi}^{Q}:D_{n,g}\to\mathcal{K}(F_{+})$
by
\[
\bar{\varphi}^{Q}((x_{1},\dots,x_{n})):=\begin{cases}
\varphi^{Q}((x_{1},\dots,x_{n})) & x_{1},\dots,x_{n}\ \text{are in general position}\\
0 & {\rm otherwise}.
\end{cases}
\]
Fix $x\in i_{\emptyset}^{-1}(g)$. Then it is known that 
\[
\sum_{\epsilon\in E}\epsilon\mathcal{L}(\bar{\varphi}^{Q}(F(x)))={\rm sgn}(g)\bm{1}_{i_{\emptyset}^{-1}(g)}.
\]
and $\bar{\varphi}^{Q}(a)=0$ for $a\in\ker(D_{n,g}\to D_{n-1,g})$
\cite{MR3198753}\cite{MR3351752}. Let $z\in C_{n}(U_{g})$ be a
lift of $\eta_{E}$. Since 
\[
F(x)-z\in I_{E}D_{n,g}+\ker(D_{n,g}\to D_{n-1,g}),
\]
we have
\[
\sum_{\epsilon\in E}\epsilon\mathcal{L}(\bar{\varphi}^{Q}(F(x)-z))=0.
\]
Thus we have
\[
\sum_{\epsilon\in E}\epsilon\mathcal{L}(\bar{\varphi}^{Q}(z))={\rm sgn}(g)\bm{1}_{i_{\emptyset}^{-1}(g)}.
\]
Since $\bar{\varphi}^{Q}(z)=\varphi^{Q}(z)=\varphi_{E}(z)$ from Lemma
\ref{lem:QE}, $\bar{\varphi}^{Q}(z)-a\in I_{E}Z(U_{g})$. Thus the
claim is proved.
\end{proof}
For a subgroup $E'$ of $E$ of finite index, we have
\[
\alpha_{E,E'}(\eta_{E'})=\#(E/E')\cdot\eta_{E}.
\]

\begin{defn}
Let $E$ be a subgroup of finite index of $\mathcal{O}_{F,+}^{\times}$.
Let $U$ be a subset of $F^{\times}$ such that $U\cap i_{\emptyset}^{-1}(g)$
is fat for all $g\in X_{\emptyset}$. Assume that $U$ is closed under
the action of $E$. For $g\in X_{\emptyset}$, we define $\vartheta_{\infty,g}(U,E)\in H_{0}(E,Z(U_{g}))$
by
\[
\vartheta_{\infty,g}(U,E)={\rm sgn}(g)\Omega_{E,U_{g}}(\eta_{E}).
\]

\end{defn}

\begin{defn}
\label{Def:s1_theta_inf}Let $E$ be a subgroup of finite index of
$\mathcal{O}_{F}^{\times}$. Let $U$ be a subset of $F^{\times}$
such that $U\cap i_{\emptyset}^{-1}(g)$ is fat for all $g\in X_{\emptyset}$.
Assume that $U$ is closed under the action of $E$. Put $E_{+}=E\cap\mathcal{O}_{F,+}^{\times}$.
Let $f_{g}:Z(U_{g})\to\mathcal{Z}(U,\emptyset)$ be the natural inclusion.
Put 
\[
a_{g}=\alpha_{E_{+},E}(f_{g}(\vartheta_{\infty,g}(U,E_{+})))\in H_{0}(E,\mathcal{Z}(U,W)).
\]
Note that $a_{\epsilon g}=a_{g}$ for all $g\in X_{\emptyset}$ and
$\epsilon\in E$. We define $\vartheta_{\infty}(U,E)\in H_{0}(E,\mathcal{Z}(U,W))$
by
\[
\vartheta_{\infty}(U,E)=\sum_{g\in X_{\emptyset}/(E/E_{+})}a_{g}.
\]
\end{defn}
\begin{prop}
\label{prop:D_is_signed_fd}Let $D\in\mathcal{B}_{\infty}(0)$ be
a lift of $\vartheta_{\infty}(U,E)$. Then we have
\[
\sum_{\epsilon\in E}\epsilon\mathcal{L}_{\infty}(D)=\bm{1}_{F^{\times}}.
\]
\end{prop}
\begin{proof}
It follows from Lemma 
\end{proof}

\begin{prop}
\label{prop:theta_inf_vanish}Fix $v\in S_{\infty}$. Let $h:H_{0}(E,\mathcal{Z}(U,\emptyset))\to H_{0}(E,\mathcal{Z}(U,\{v\}))$
be the map induced from the natural map$\mathcal{Z}(U,\emptyset)\to\mathcal{Z}(U,\{v\})$.
Then we have
\[
h(\vartheta_{\infty}(U,E))=0.
\]
\end{prop}
\begin{proof}
The case $(F,E)=\{1\}$ is obvious from the definition. Therefore
we assume that $F\neq Q$. Put $E_{+}=E\cap\mathcal{O}_{F,+}^{\times}$
and $E'=E\cap i_{\{v\}}^{-1}(1)$. For $g\in X_{\emptyset}$, consider
the commutative diagram
\[
\xymatrix{H_{n-1}(E_{+},\mathbb{Z}^{-})\ar[r]^{\Omega_{E_{+},U_{g}}}\ar[rd]_{\Omega_{E_{+},U_{q(g)}}} & H_{0}(E_{+},Z(U_{g}))\ar[r]^{f_{g}}\ar[d] & H_{0}(E_{+},\mathcal{Z}(U,\emptyset))\ar[r]^{\alpha_{E_{+},E}}\ar[d] & H_{0}(E,\mathcal{Z}(U,\emptyset))\ar[d]^{h}\\
 & H_{0}(E_{+},Z(U_{q(g)}))\ar[r]^{f_{q(g)}} & H_{0}(E_{+},\mathcal{Z}(U,\{v\}))\ar[r]^{\alpha_{E_{+},E}} & H_{0}(E,\mathcal{Z}(U,\{v\}))
}
\]
where $q(g)\in X_{\{v\}}$ is the image of $g$, and $f_{g},f_{q(g)}$
are natural maps. From the definition, we have
\[
\vartheta_{\infty}(U,E)=\sum_{g\in X_{\infty}/(E/E_{+})}a_{g}
\]
where we put
\[
a_{g}:={\rm sgn}(g)\times\alpha_{E_{+},E}\circ f_{g}\circ\Omega_{E+,U_{g}}(\eta_{E_{+}}).
\]
We have 
\begin{equation}
h(a_{g})={\rm sgn}(g)\times\alpha_{E_{+},E}\circ f_{q(g)}\circ\Omega_{E_{+},U_{q(g)}}(\eta_{E_{+}}).\label{eq:e1}
\end{equation}

(The case $E_{+}=E'$). Let $e$ be a unique element of $i_{\{v\}}^{-1}(1)\setminus\{1\}$.
Then we have
\[
\vartheta_{\infty}(U,E)=\sum_{g\in X_{\infty}/(E/E_{+})/\{1,e\}}(a_{g}+a_{ge}).
\]
 Since $q(ge)=q(g)$ and ${\rm sgn}(ge)=-{\rm sgn}(g)$, we have
\[
h(a_{g}+a_{ge})=0\ \ \ (g\in X_{\emptyset}).
\]
Thus we have
\[
h(\vartheta_{\infty}(U,E))=0.
\]

(The case $E_{+}=E'$). Let us consider the commutative diagram
\[
\xymatrix{H_{n-1}(E_{+},\mathbb{Z}^{-})\ar[r]_{\Omega_{E_{+},U_{q(g)}}}\ar[d]_{\alpha_{E_{+},E}} & H_{0}(E_{+},Z(U_{q(g)}))\ar[r]^{f_{q(g)}}\ar[d]^{\alpha_{E_{+},E}} & H_{0}(E_{+},\mathcal{Z}(U,\{v\}))\ar[d]_{\alpha_{E_{+},E'}}\ar[rd]^{\alpha_{E_{+},E}}\\
H_{n-1}(E',\mathbb{Z}^{-})\ar[r]_{\Omega_{E',U_{q(g)}}} & H_{0}(E',Z(U_{q(g)}))\ar[r] & H_{0}(E',\mathcal{Z}(U,\{v\}))\ar[r]_{\alpha_{E',E}} & H_{0}(E,\mathcal{Z}(U,\{v\})).
}
\]
We have $H_{n-1}(E',\mathbb{Z}^{-})=0$ from the assumption $F\neq\mathbb{Q}$.
Thus we have 
\[
h(a_{g})=\alpha_{E_{+},E}\circ f_{q(g)}\circ\Omega_{E_{+},U_{q(g)}}(\eta_{E_{+}})=0
\]
for $g\in X_{\emptyset}$. Hence we have $h(\vartheta_{\infty}(U,E))=0$.\end{proof}
\begin{prop}
\label{prop:theta_inf_mul_p}Fix $p\in F^{\times}$. Assume that $pU\subset U$.
Then we have
\[
[p]\vartheta_{\infty}(U,E)=\vartheta_{\infty}(U,E).
\]
\end{prop}
\begin{proof}
Let $(a_{g})_{g\in X_{\emptyset}}$ be as in Definition \ref{Def:s1_theta_inf}.
Take $g\in X_{\emptyset}$. Let us consider the diagram
\[
\xymatrix{H_{n-1}(E_{+},\mathbb{Z}^{-})\ar[r]^{\Omega_{E_{+},U_{g}}}\ar[rd]_{\Omega_{E_{+},U_{pg}}} & H_{0}(E_{+},Z(U_{g}))\ar[r]^{f_{g}}\ar[d]^{\times{\rm sgn}(p)[p]} & H_{0}(E_{+},\mathcal{Z}(U,\emptyset))\ar[r]^{\alpha_{E_{+},E}}\ar[d]^{\times{\rm sgn}(p)[p]} & H_{0}(E,\mathcal{Z}(U,\emptyset))\ar[d]^{\times{\rm sgn}(p)[p]}\\
 & H_{0}(E_{+},Z(U_{pg}))\ar[r]_{f_{pg}} & H_{0}(E_{+},\mathcal{Z}(U,\emptyset))\ar[r]_{\alpha_{E_{+},E}} & H_{0}(E,\mathcal{Z}(U,\emptyset))
}
\]
where $f_{g}$ and $f_{pg}$ are natural maps. From Lemma \ref{lem:s1_V_ptimes},
we have 
\[
\Omega_{E_{+},U_{pg}}=(\times{\rm sgn}(p)[p])\circ\Omega_{E_{+},U_{pg}}.
\]
Thus this diagram commutes. From the definition, we have
\begin{eqnarray*}
a_{g} & = & {\rm sgn}(g)\times\alpha_{E_{+},E}\circ f_{g}\circ\Omega_{E_{+},U_{g}}(\eta_{E})\\
a_{pg} & = & {\rm sgn}(pg)\times\alpha_{E_{+},E}\circ f_{pg}\circ\Omega_{E_{+},U_{pg}}(\eta_{E}).
\end{eqnarray*}
Since the diagram commutes, we have
\[
[p]a_{g}=a_{pg}.
\]
Hence we have
\[
[p]\vartheta_{\infty}(U,E)=[p]\sum_{g}a_{g}=\sum_{g}a_{pg}=\vartheta_{\infty}(U,E).
\]

\end{proof}

\section{Shintani zeta functions\label{sec:ShintaniZeta}}

In this section, we view some basic definitions and results concerning
to the Shintani zeta functions. This section does not contain any
new result. We denote by $\mathbb{A}_{{\rm finite}}$ the finite adele
ring of $\mathbb{Q}$.

\subsection{The Shintani zeta functions and Solomon-Hu parings}

Let $V$ be an $n$-dimensional vector space over $\mathbb{Q}$ and
$\mathcal{U}$ a subset of $V$. We say that $\Phi\in\mathcal{S}(V\otimes\mathbb{A}_{{\rm finite}})$
is $\mathcal{U}$-smooth if
\[
\int_{\mathbb{A}_{{\rm finite}}}\Phi(x+tu)dt=0\ \ \ \ \forall(x,u)\in V\times\mathcal{U}.
\]
We define the action of $\mathbb{Z}[V]$ to ${\rm Map}(V,\mathbb{Z})$
by $([y]f)(x)=f(x-y)$. Let $\epsilon:\mathbb{Z}[V]\to\mathbb{Z}$
be the augmentation map. We view $\mathbb{Q}$ as a $\mathbb{Z}[V]$-module
by the action
\[
\Delta m=\epsilon(\Delta)m\ \ \ \ (\Delta\in\mathbb{Z}[V],m\in\mathbb{Q}).
\]
We denote by $\delta\in{\rm Map}(V,\mathbb{Z})$ the Kronecker delta
function, i.e., $\delta(0)=1$ and $\delta(x)=0$ for $x\neq0$. We
denote by $A'(V)$ the set of maps from $V$ to $\mathbb{Z}$ which
has a finite support. We put
\[
A(V)=\{f:V\to\mathbb{Z}\mid\exists\Delta\in X(V),\ \epsilon(\Delta)\neq0\ {\rm and}\ \Delta f\in A'(V)\}.
\]

\begin{lem}
For $f\in\mathcal{K}'(\mathcal{U})$ and $\mathcal{U}$-smooth $\Phi$,
we have $f\cdot\Phi\in A(V)$\end{lem}
\begin{proof}
It is enough to only consider the case $f=\bm{1}_{C(u_{1},\dots,u_{m})}$
where $u_{1},\dots,u_{m}\in\mathcal{U}$. There exists a positive
integer $k$ such that
\begin{eqnarray*}
{\rm supp}(\Phi)\cap(\mathbb{Q}u_{1}+\cdots+\mathbb{Q}u_{m}) & \subset & \frac{1}{k}\mathbb{Z}u_{1}+\cdots+\frac{1}{k}\mathbb{Z}u_{m}\\
\Phi(x+ku_{j}) & = & \Phi(x)\ \ \ \ (x\in V,\ j=1,\dots,m).
\end{eqnarray*}
If we put $\Delta=\prod_{j=1}^{m}\sum_{c=0}^{k^{2}-1}[cu_{j}/k]\in X(V)$
then $\epsilon(\Delta)=k^{2m}\neq0$ and $\Delta(f\cdot\Phi)\in A'(V)$.
\end{proof}
Let $\varpi:A(V)\to\mathbb{Q}$ be the unique $\mathbb{Z}[V]$-homomorphism
such that $\varpi(\delta)=1$. For $f\in\mathcal{K}(\mathcal{U})$
and $\mathcal{U}$-smooth $\Phi$, we define a (constant term of)
Solomon-Hu pairing by $\langle\!\langle f,\Phi\rangle\!\rangle=\varpi(f\cdot\Phi)$.
The importance of this pairing comes from the following Lemma:
\begin{lem}
Let $m$ be a positive integer and $\rho_{1},\dots,\rho_{m}$ homomorphisms
from $V$ to $\mathbb{R}$ such that $\rho_{j}(u)>0$ for all $u\in\mathcal{U}$
and $j=1,\dots,m$. Take $f\in\mathcal{K}(\mathcal{U})$ and $\mathcal{U}$-smooth
$\Phi$. Then the series
\[
\zeta((s_{1},\dots,s_{m}),\Phi,f):=\sum_{f(x)\neq0}f(x)\Phi(x)\rho_{1}(x)^{-s_{1}}\cdots\rho_{m}(x)^{-s_{m}}
\]
converges absolutely if the real part of $s_{1}+\cdots+s_{m}$ are
larger than $n$. Moreover $\zeta((s_{1},\dots,s_{m}),\Phi,f)$ is
analytically continued to the whole $\mathbb{C}^{m}$ and satisfies
\[
\zeta((0,\dots,0),\Phi,f)=\langle\!\langle f,\Phi\rangle\!\rangle.
\]
\end{lem}
\begin{proof}
The absolute convergence is obvious. Let $M$ the field of meromorphic
functions on $\mathbb{C}^{m}$. We view $M$ as a $\mathbb{Z}[V]$-module
by the action
\[
([y]f)(t_{1},\dots,t_{m})=\exp(-\rho_{1}(y)t_{1}-\cdots-\rho_{m}(y)t_{m})\cdot f(t_{1},\dots,t_{m})\ \ \ \ (f\in M).
\]
Let $\Omega:A(V)\to M$ be the unique $\mathbb{Z}[V]$-homomorphism
such that $\Omega(\delta)=1$. Put $h=\Omega(f\cdot\Phi)\in M$. Then
we have
\begin{multline*}
\zeta((s_{1},\dots,s_{m}),\Phi,f)\\
=\prod_{j=1}^{m}\frac{1}{\Gamma(s_{j})}\int_{0}^{\infty}\cdots\int_{0}^{\infty}h(t_{1},\dots,t_{m})t_{1}^{s_{1}-1}\cdots t_{m}^{s_{m}-1}dt_{1}\cdots dt_{m}
\end{multline*}
if the real parts of $s_{1},\dots,s_{m}$ are large enough. Thus
\[
\prod_{j=1}^{m}\frac{1}{\Gamma(s_{j})({\rm e}^{2\pi is_{j}}-1)}\int_{C}\cdots\int_{C}h(t_{1},\dots,t_{m})t_{1}^{s_{1}-1}\cdots t_{m}^{s_{m}-1}dt_{1}\cdots dt_{m}
\]
gives an analytic continuation of $\zeta((s_{1},\dots,s_{m}),\Phi,f)$
where $C$ is the Hankel contour. From the simple residue calculus,
we obtain
\[
\zeta((0,\dots,0),\Phi,f)=h(0,\dots,0)=\langle\!\langle f,\Phi\rangle\!\rangle.
\]

\end{proof}

\subsection{Some integrality results concerning Solomon-Hu parings}

Let $q$ be a rational prime. Fix two $\mathbb{Z}_{q}$-lattices $\Gamma'\subset\Gamma\subset V\otimes\mathbb{Q}_{q}$.
In this section, we consider the case $\mathcal{U}=(\Gamma\setminus\Gamma')\cap V$.
For $x\in\Gamma/\Gamma'$, define $h_{x}\in\mathcal{S}(V\otimes\mathbb{Q}_{q})$
by $h_{x}=\bm{1}_{\pi^{-1}(x)}-\bm{1}_{\pi^{-1}(0)}$ where $\pi:\Gamma\to\Gamma/\Gamma'$
is the projection. We write $S_{0}$ for the submodule of $\mathcal{S}(V\otimes\mathbb{Q}_{q})$
spanned by $\{h_{x}\mid x\in\Gamma/\Gamma'\}$, and $S_{1}$ for the
submodule of $\mathcal{S}(V\otimes\mathbb{Q}_{q})$ spanned by $\sum_{x\in\Gamma/\Gamma'}h_{x}$. 
\begin{lem}
\label{lem:weakIntegrality}Let $\Phi\in\mathcal{S}(V\otimes\prod_{p\neq q}'\mathbb{Q}_{p})\otimes S_{0}\subset\mathcal{S}(V\otimes\mathbb{A}_{{\rm finite}})$
and $f\in\mathcal{K}(\mathcal{U})$. Then $f\cdot\Phi\in A(V)$ and
\[
\langle\!\langle f,\Phi\rangle\!\rangle\in\frac{1}{q^{n}}\mathbb{Z}.
\]
\end{lem}
\begin{proof}
It is enough to prove that $\varpi(\bm{1}_{C(u_{1},\dots,u_{m})}\cdot(\Phi'\otimes\Phi''))\in\frac{1}{q^{n}}\mathbb{Z}$
for all $u_{1},\dots,u_{m}\in\mathcal{U}$, $\Phi'\in\mathcal{S}(V\otimes\prod_{p\neq q}'\mathbb{Q}_{p})$
and $\Phi''\in S_{0}$. Put $h=\bm{1}_{C(u_{1},\dots,u_{m})}\cdot(\Phi'\otimes\Phi'')$.
Take a positive integer $k\in\mathbb{Z}\setminus q\mathbb{Z}$ such
that $\Phi'$ is periodic with respect to $\sum_{i=1}^{m}\mathbb{Z}ku_{m}$.
Put $\Delta=\prod_{j=1}^{m}\sum_{c=0}^{q-1}[cku_{j}]\in\mathbb{Z}[V]$.
Then $\Delta h\in A'(V)$. Since $\epsilon(\Delta)=q^{m}$, the lemma
is proved.
\end{proof}
The following lemma was proved in \cite[Proposition 4.1]{DasguptaSpiess}.
\begin{lem}
\label{lem:strongIntegrality}Let $\Phi\in\mathcal{S}(V\otimes\prod_{p\neq q}'\mathbb{Q}_{p})\otimes S_{1}\subset\mathcal{S}(V\otimes\mathbb{A}_{{\rm finite}})$
and $f\in\mathcal{K}(\mathcal{U})$. If $q\geq n+2$ then 
\[
\langle\!\langle f,\Phi\rangle\!\rangle\in\mathbb{Z}.
\]
\end{lem}
\begin{proof}
It is enough to prove that $\varpi(\bm{1}_{C(u_{1},\dots,u_{m})}\cdot(\Phi'\otimes\Phi''))\in\frac{1}{q^{n}}\mathbb{Z}$
for all $u_{1},\dots,u_{m}\in\mathcal{U}$, $\Phi'\in\mathcal{S}(V\otimes\prod_{p\neq q}'\mathbb{Q}_{p})$
and $\Phi''=q\bm{1}_{\Gamma'}-\bm{1}_{\Gamma}$. Put $h=\bm{1}_{C(u_{1},\dots,u_{m})}\cdot(\Phi'\otimes\Phi'')$.
Take a positive integer $k\in\mathbb{Z}\setminus q\mathbb{Z}$ such
that $\Phi'$ is periodic with respect to $\sum_{j=1}^{m}\mathbb{Z}ku_{j}$.
Put $w_{j}=ku_{j}$ for $j=1,\dots,m$. Let $\zeta_{q}$ be a primitive
$q$-th root of unity. We denote by $M$ the set of non-trivial group
homomorphism from $\Gamma/\Gamma'$ to $\mathbb{Z}[\zeta_{q}]^{\times}$.
For $\varphi\in M$, we put
\[
h_{\varphi}:=\sum_{\varphi\in M}\bm{1}_{C(u_{1},\dots,u_{m})}\cdot(\Phi'\otimes\varphi).
\]
Then we have $h=\sum_{\varphi\in M}h_{\varphi}$. For $\varphi\in M$,
we have
\[
\varpi(h_{\varphi})\in(1-\zeta_{q})^{-m}\mathbb{Z}[\zeta_{q}]
\]
since $\Delta_{\varphi}h_{\varphi}\in A'(V)\otimes\mathbb{Z}[\zeta_{q}]$
and $1/\epsilon(\Delta_{\varphi})\in(1-\zeta_{q})^{-m}\mathbb{Z}[\zeta_{q}]$
where
\[
\Delta_{\varphi}:=\prod_{j=1}^{m}(1-\varphi(w_{j})[w_{j}]).
\]
 Hence, from the condition $q\geq n+2$, we have 
\[
\varpi(h)=\sum_{\varphi\in M}\varpi(h_{\varphi})\in(1-\zeta_{q})^{-m}\mathbb{Z}[\zeta_{q}]\cap\mathbb{Q}=\mathbb{Z}.
\]

\end{proof}

\section{Construction of the Shintani data ${\rm Sh}^{v_{0}}$, ${\rm Sh}^{\diamond}$,
${\rm Sh}^{v_{0},\diamond}$\label{sec:constSh}}

For $W\subset S$, let $\mathbb{R}[[(s_{v})_{v\in S\setminus W}]]$
be the formal power series ring in $\#(S\setminus W)$ variables.
We regard $\mathbb{R}[[(s_{v})_{v\in S\setminus W}]]$ as a $\mathbb{Z}[F^{\times}]$-module
by the action 
\[
[x]p=p\prod_{v\in S\setminus W}\left|x\right|_{v}^{-s_{v}}\ \ \ \ \ (x\in F^{\times},p\in\mathbb{R}[[(s_{v})_{v\in S\setminus W}]]).
\]
We denote by $\mathbb{P}_{W}$ the submodule of $\mathbb{R}[[(s_{v})_{v\in S\setminus W}]]$
consisting of formal power series whose constant terms are in $\mathbb{Z}$.
Let $\mathcal{R}$ be the functor from ${\bf Sub}(S)$ to ${\bf Mod}(\mathbb{Z}[F^{\times}])$
defined by $\mathcal{R}(W):=\mathbb{P}_{W}\otimes\mathbb{Z}[N_{F}/\prod_{v\in W}N_{v}]$.
Put $\Upsilon=\mathbb{R}[[s]]\otimes\mathbb{Z}[G]$. We define $\lambda:\mathcal{R}(\emptyset)\to\Upsilon$
by 
\[
\lambda(P(s_{v_{0}},\cdots,s_{v_{r}})\otimes[u])=P(s,\dots,s)\prod_{v\notin S}\left|u_{v}\right|_{v}^{-s}\otimes[{\rm rec}(u)].
\]
Put $\theta=\Theta_{S,\{\mathfrak{q}\},K}(s)\in\Upsilon$.

Let us fix a primed ideal $\mathfrak{q}\notin S$. Put $q={\rm ch}(\mathfrak{q})$.
We put $\mathbb{F}_{q}=\mathbb{Z}/q\mathbb{Z}\subset\kappa_{\mathfrak{q}}$.
The purpose of this section is to construct three Shintani data ${\rm Sh}^{v_{0}},{\rm Sh}^{\diamond},{\rm Sh}^{v_{0},\diamond}$
for $(\mathcal{R},\Upsilon,\lambda,\theta)$. First, we define a triple
$(\mathcal{B}_{R,M},\mathcal{L}_{R,M},\vartheta_{R,M})$ for a subset
$R$ of $S_{f}$ and a subgroup $M$ of $\mathbb{F}_{q}^{\times}$.
We define ${\rm Sh}^{v_{0}},{\rm Sh}^{\diamond},{\rm Sh}^{v_{0},\diamond}$
by using $(\mathcal{B}_{R,M},\mathcal{L}_{R,M},\vartheta_{R,M})$
for $(R,M)=(S_{f}\cap\{v_{0}\},\mathbb{F}_{q}^{\times}),(\emptyset,1),(S_{f}\cap\{v_{0}\},1)$
respectively.

\subsection{Definition of $\mathcal{A}_{R,G}(W)$}

Fix a subset $R$ of $S_{f}$ and a subgroup $M$ of $\mathbb{F}_{q}^{\times}$.
We put $\tilde{M}:=\ker(O_{\mathfrak{q}}^{\times}\to\kappa_{\mathfrak{q}}^{\times}\to\kappa_{\mathfrak{q}}^{\times}/M)$.
\begin{defn}
\label{Def:hx}For $x\in\mathbb{F}_{q}\subset\kappa_{\mathfrak{q}}$,
define $h_{x}\in\mathcal{S}(F_{\mathfrak{q}})$ by $h_{x}=\bm{1}_{\pi^{-1}(x)}-\bm{1}_{\pi^{-1}(0)}$
where $\pi:O_{\mathfrak{q}}\to O_{\mathfrak{q}}/\mathfrak{q}O_{\mathfrak{q}}=\kappa_{\mathfrak{q}}$
is the natural composite map.
\end{defn}
For a prime ideal $\mathfrak{p}$, we put $V_{\mathfrak{p}}=1+\mathfrak{p}^{m}O_{\mathfrak{p}}$
where $m$ is a least positive integer such that $-1\notin1+\mathfrak{p}^{m}O_{\mathfrak{p}}$.

\begin{defn}
For $W\subset S_{f}\setminus R$, define $f_{R,M,W}\in\mathcal{S}(\mathbb{A}_{F}^{S_{\infty}})$
by
\[
f_{R,M,W}(x):=\prod_{\mathfrak{p}\notin S_{\infty}}f_{\mathfrak{p}}(x_{\mathfrak{p}})
\]
where $f_{\mathfrak{p}}\in\mathcal{S}(F_{\mathfrak{p}})$ is define
by
\[
f_{\mathfrak{p}}=\begin{cases}
\sum_{x\in M}h_{x} & \mathfrak{p}=\mathfrak{q}\\
\bm{1}_{V_{\mathfrak{p}}} & \mathfrak{p}\in R\\
\bm{1}_{O_{\mathfrak{p}}^{\times}} & \mathfrak{p}\in S_{f}\setminus(W\cup R)\\
\bm{1}_{\mathfrak{p}O_{\mathfrak{p}}} & \mathfrak{p}\in W\\
\bm{1}_{O_{\mathfrak{p}}} & \mathfrak{p}\notin S_{f}\cup\{\mathfrak{q}\}.
\end{cases}
\]

\end{defn}
For $x\in\mathbb{A}_{F}^{S_{\infty},\times}$ and $f\in\mathcal{S}(\mathbb{A}_{F}^{S_{\infty}})$,
define $xf\in\mathcal{S}(\mathbb{A}_{F}^{S_{\infty}})$ by $(xf)(xy)=f(y)$.
The function $f_{R,M,W}$ satisfies the following properties.
\begin{itemize}
\item $f_{R,M,W}$ is invariant under the action of $\prod_{\mathfrak{p}\notin R\cup\{\mathfrak{q}\}}O_{\mathfrak{p}}^{\times}\times\prod_{\mathfrak{p}\in R}V_{\mathfrak{p}}\times\tilde{M}$.
\item For $\mathfrak{p}\in S_{f}\setminus(W\cup R)$, 
\begin{equation}
\sum_{x\in O_{\mathfrak{p}}^{\times}/V_{\mathfrak{p}}}xf_{\{\mathfrak{p}\}\cup R,M,W}=f_{R,M,W}\label{eq:fRMW_Rchange}
\end{equation}

\item For $M_{1}\subset M_{2}$, 
\begin{equation}
\sum_{x\in\tilde{M}_{2}/\tilde{M}_{1}}xf_{R,M_{1},W}=f_{R,M_{2},W}.\label{eq:fRMW_Mchange}
\end{equation}

\item For $\mathfrak{p}\in S_{f}/(R\cup W)$,
\begin{equation}
(\pi_{\mathfrak{p}}^{-1}-1)f_{R,M,W\cup\{\mathfrak{p}\}}=f_{R,M,W}\label{eq:fRMW_Wchange}
\end{equation}

\end{itemize}

\begin{defn}
For $W\subset S_{f}\setminus R$, we denote by $\mathcal{A}_{R,M}(W)$
the subgroup of $\mathcal{S}(\mathbb{A}_{F}^{S_{\infty}})$ spanned
by $\{xf_{R,M,W}\mid x\in\mathbb{A}_{F}^{S_{\infty}\cup S_{f}\cup S_{q},\times}\}$. 
\end{defn}

\begin{defn}
We define subgroups $F_{R,M},\mathcal{E}_{R,M}$ and a submonoid $\mathcal{E}_{R,M}^{\circ}$
of $F^{\times}$ by
\begin{align*}
F_{R,M}:= & F^{\times}\cap(\tilde{M}\times\prod_{\mathfrak{p}\in R}V_{\mathfrak{p}})\\
\mathcal{E}_{R,M}:= & F_{R,M}^{\times}\cap\left(\bigcap_{\mathfrak{p}\in S_{f}\cup S_{q}}\mathcal{O}_{\mathfrak{p}}^{\times}\right)\\
\mathcal{E}_{R,M}^{\circ}:= & F_{R,M}^{\times}\cap\left(\bigcap_{\mathfrak{p}\in S_{q}}\mathcal{O}_{\mathfrak{p}}\right).
\end{align*}
We have $\mathcal{E}_{R,M}\subset\mathcal{E}_{R,M}^{\circ}\subset F_{R,M}\subset F^{\times}$.
Note that $\mathcal{A}_{R,M}(W)$ is isomorphic to a direct sum of
copies of $\mathbb{Z}[\mathcal{E}_{R,M}/(\mathcal{E}_{R,M}\cap\mathcal{O}_{F}^{\times})]$.
\end{defn}

\subsection{Definition of the functor $\mathcal{B}_{R,M}$}

Fix a subset $R$ of $S_{f}$ and a subgroup $M$ of $\mathbb{F}_{q}^{\times}$.
In this section, we define a functor $\mathcal{B}_{R,M}$ from ${\bf Sub}(S\setminus R)$
to ${\bf Sub}(\mathbb{Z}[F^{\times}])$. We denote by $\mathcal{U}$
the set of $x\in F^{\times}$ which satisfy the following conditions.
\begin{itemize}
\item For all $\mathfrak{p}\in S_{q}\setminus\{\mathfrak{q}\}$, ${\rm ord}_{\mathfrak{p}}(x)\geq\min\{m\in\mathbb{Z}_{\geq1}\mid1+\pi_{\mathfrak{p}}^{m}O_{\mathfrak{p}}\subset J_{\mathfrak{p}}\}$.
\item $x_{\mathfrak{q}}\in\pi^{-1}(\mathbb{F}_{q}^{\times})$ where $\pi:O_{\mathfrak{q}}\to\kappa_{\mathfrak{q}}$
is the natural projection.
\end{itemize}
Note that $\mathcal{U}$ is closed under the action of $\mathcal{E}_{R,M}^{\circ}$.
For $W\subset S_{\infty}$, we put
\[
\mathcal{Q}(W):=\mathcal{S}(\mathbb{A}_{F}^{S_{\infty}})\otimes\mathcal{Z}(F^{\times},W)\otimes\mathbb{Z}[N^{S}].
\]

\begin{defn}
For $W\subset S\setminus R$, we define a $\mathbb{Z}[\mathcal{E}_{R,M}]$-submodule
$\mathcal{B}_{R,M}^{*}(W)$ of $\mathcal{Q}(W\cap S_{\infty})$ by 

\[
\mathcal{B}_{R,M}^{*}(W):=\mathcal{A}_{R,M}(W\cap S_{f})\otimes\mathcal{Z}(\mathcal{U},W\cap S_{\infty})\otimes\mathbb{Z}[N^{S}].
\]

\end{defn}

\begin{defn}
For $W\subset S\setminus R$, we define a $\mathbb{Z}[F^{\times}]$-module
$\mathcal{B}_{R,M}(W)$ by 
\end{defn}
\[
\mathcal{B}_{R,M}(W):=\mathcal{B}_{R,M}^{*}(W)\otimes_{\mathbb{Z}[\mathcal{E}_{R,M}]}\mathbb{Z}[F^{\times}].
\]
 Note that the natural map 
\[
\mathcal{B}_{R,M}^{*}(W)\otimes_{\mathbb{Z}[\mathcal{E}_{R,M}]}\mathbb{Z}[F^{\times}]\to\mathcal{Q}(W\cap S_{\infty})
\]
is injective since 
\[
\mathcal{A}_{R,M}(W)\otimes_{\mathbb{Z}[\mathcal{E}_{R,M}]}\mathbb{Z}[F^{\times}]\to\mathcal{S}(\mathbb{A}_{F}^{S_{\infty}})
\]
is injective. From now we regard $\mathcal{B}_{R,M}(W)$ as a $\mathbb{Z}[F^{\times}]$-submodule
of $\mathcal{Q}(W\cap S_{\infty})$.
\begin{prop}
\label{Prop:Bw1Bw2}Let $W_{1}\subset W_{2}\subset S\setminus R$.
Let $j:\mathcal{Q}(W_{1}\cap S_{\infty})\to\mathcal{Q}(W_{2}\cap S_{\infty})$
be the natural map. Then we have
\[
j(\mathcal{B}_{R,M}(W_{1}))\subset\mathcal{B}_{R,M}(W_{2}).
\]
\end{prop}
\begin{proof}
It is enough to consider the case $W_{1}=W_{2}\setminus\{v\}$. If
$v$ is infinite then the claim is obvious. We assume that $v$ is
finite. It is enough to prove that $xf_{R,M,W_{1}}\otimes b\otimes c\subset\mathcal{B}_{R,M}(W_{2})$
for $x\in\mathbb{A}_{F}^{S_{\infty}\cup S_{f}\cup S_{q},\times}$,
$b\in\mathcal{Z}(\mathcal{U},W_{1}\cap S_{\infty})$, and $c\in\mathbb{Z}[N^{S}]$.
Let $p$ be an any element of $F_{R,M}\cap(\pi_{v}O_{v}^{\times})\cap\mathcal{O}_{S_{\mathfrak{q}}\cup S_{f}\setminus\{v\}}^{\times}\neq0$.
Then we have $pb\in b\in\mathcal{Z}(\mathcal{U},W_{1}\cap S_{\infty})$
since $p\in\mathcal{E}_{R,M}^{\circ}$. Thus we have 
\begin{align*}
xf_{R,M,W_{1}}\otimes b\otimes c & =x(\pi_{v}^{-1}-1)f_{R,M,W_{2}}\otimes b\otimes c\\
 & =[p^{-1}]\left(xp\pi_{v}^{-1}f_{R,M,W_{2}}\otimes pb\otimes pc\right)-xf_{R,M,W_{2}}\otimes b\otimes1.
\end{align*}
Since $\mathcal{A}_{R,M,W_{2}}$ is invariant under the action of
$xp\pi_{v}^{-1}$, we have $xp\pi_{v}^{-1}f_{R,M,W_{2}}\in\mathcal{A}_{R,M}(W_{2})$.
Thus the claim is proved.
\end{proof}
From Proposition \ref{Prop:Bw1Bw2}, we can regard $\mathcal{B}_{R,M}$
as a functor from ${\bf Sub}(S\setminus R)$ to ${\bf Mod}(\mathbb{Z}[F^{\times}])$.
\begin{lem}
\label{lem:homology_forget}Let $A$ be an abelian group, $A'$ a
subgroup of $A$, and $X$ a $\mathbb{Z}[A]$-module. Then we have
\[
H_{i}(A,X\otimes\mathbb{Z}[A/A'])=H_{i}(A',X).
\]
\end{lem}
\begin{proof}
Put $Y:={\rm Fl}(X)\otimes_{\mathbb{Z}[A']}\mathbb{Z}[A]$ where ${\rm Fl}$
is the forgetful functor from ${\bf Sub}(\mathbb{Z}[A])$ to ${\bf Sub}(\mathbb{Z}[A'])$.
Define $f\in{\rm Hom}_{\mathbb{Z}[A]}(Y,X\otimes\mathbb{Z}[A/A'])$
by
\[
f(x\otimes a)=ax\otimes a\ \ \ \ (x\in{\rm Fl}(X),a\in\mathbb{Z}[A]).
\]
Then $f$ is an isomorphism. Thus we have
\begin{align*}
H_{i}(A,X\otimes\mathbb{Z}[A/A']) & =H_{i}(A,Y)\\
 & =H_{i}(A',{\rm Fl}(X)).
\end{align*}
Thus the lemma is proved.\end{proof}
\begin{prop}
\label{prop:homology_vanish}Let $E$ be a subgroup of $F^{\times}$,
and $W$ a subset of $S\setminus R$. If $-1\notin F_{R,M}\cap E$
or $S_{\infty}\neq W\cap S_{\infty}$, we have 
\[
H_{i}(E,\mathcal{B}_{R,M}(W))=0
\]
for $i>0$.\end{prop}
\begin{proof}
Since 
\begin{align*}
\mathcal{B}_{R,M}(W) & =\mathcal{B}_{R,M}^{*}(W)\otimes_{\mathbb{Z}[\mathcal{E}_{R,M}]}\mathbb{Z}[F^{\times}]\\
 & =(\mathcal{B}_{R,M}^{*}(W)\otimes_{\mathbb{Z}[\mathcal{E}_{R,M}]}\mathbb{Z}[\mathcal{E}_{R,M}\cdot E])\otimes_{\mathbb{Z}[\mathcal{E}_{R,M}\cdot E]}\mathbb{Z}[F^{\times}]\\
 & =(\mathcal{B}_{R,M}^{*}(W)\otimes_{\mathbb{Z}[\mathcal{E}_{R,M}\cap E]}\mathbb{Z}[E])\otimes_{\mathbb{Z}[\mathcal{E}_{R,M}\cdot E]}\mathbb{Z}[F^{\times}]
\end{align*}
 is isomorphic to a direct sum of copies of $\mathcal{B}_{R,M}^{*}(W)\otimes_{\mathbb{Z}[\mathcal{E}_{R,M}\cap E]}\mathbb{Z}[E]$
as a $\mathbb{Z}[E]$-module, it is enough to prove that
\[
H_{i}(\mathcal{E}_{R,M}\cap E,\mathcal{B}_{R,M}^{*}(W))=0.
\]
Put $U_{S}=\mathcal{E}_{R,M}\cap E\cap\mathcal{O}_{S}^{\times}$ and
$U_{F}=\mathcal{E}_{R,M}\cap E\cap\mathcal{O}_{S}^{\times}$. Since
$\mathbb{Z}[N^{S}]$ is isomorphic to a direct sum of copies of $\mathbb{Z}[(\mathcal{E}_{R,M}\cap E)/U_{S}]$
as a $\mathbb{Z}[\mathcal{E}_{R,M}\cap E]$-module, it is enough to
prove that
\[
H_{i}(\mathcal{E}_{R,M}\cap E,\mathcal{A}_{R,M}(W\cap S_{f})\otimes\mathcal{Z}(\mathcal{U},W\cap S_{\infty})\otimes\mathbb{Z}[(\mathcal{E}_{R,M}\cap E)/U_{S}])=0.
\]
From Lemma \ref{lem:homology_forget}, we have
\begin{multline*}
H_{i}(\mathcal{E}_{R,M}\cap E,\mathcal{A}_{R,M}(W\cap S_{f})\otimes\mathcal{Z}(\mathcal{U},W\cap S_{\infty})\otimes\mathbb{Z}[(\mathcal{E}_{R,M}\cap E)/U_{S}])\\
=H_{i}(U_{S},\mathcal{A}_{R,M}(W\cap S_{f})\otimes\mathcal{Z}(\mathcal{U},W\cap S_{\infty})).
\end{multline*}
Since $\mathcal{A}_{R,M}(W\cap S_{f})$ is isomorphic to a direct
sum of copies of $\mathbb{Z}[U_{S}/U_{F}]$ as a $\mathbb{Z}[U_{S}]$-module,
it is enough to prove that
\[
H_{i}(U_{S},\mathbb{Z}[U_{S}/U_{F}]\otimes\mathcal{Z}(\mathcal{U},W\cap S_{\infty}))=0.
\]
From Lemma \ref{lem:homology_forget}, we have
\[
H_{i}(U_{S},\mathbb{Z}[U_{S}/U_{F}]\otimes\mathcal{Z}(\mathcal{U},W\cap S_{\infty}))=H_{i}(U_{F},\mathcal{Z}(\mathcal{U},W\cap S_{\infty})).
\]
From Proposition \ref{prop:Hi_EZUW_vanish}, we have 
\[
H_{i}(U_{F},\mathcal{Z}(\mathcal{U},W\cap S_{\infty}))=0.
\]
Thus the claim is proved.\end{proof}
\begin{lem}
For $R_{1}\subset R_{2}\subset S_{f}$, $M_{2}\subset M_{1}\subset\mathbb{F}_{q}^{\times}$
and $W\subset S\setminus R_{2}$, we have
\[
\mathcal{B}_{R_{1},M_{1}}(W)\subset\mathcal{B}_{R_{2},M_{2}}(W).
\]
\end{lem}
\begin{proof}
It is enough to prove that $xf_{R_{1},M_{1},W\cap S_{f}}\otimes b\otimes c\subset\mathcal{B}_{R_{2},M_{2}}(W)$
for $x\in\mathbb{A}_{F}^{S_{\infty}\cup S_{f}\cup S_{q},\times}$,
$b\in\mathcal{Z}(\mathcal{U},W_{1}\cap S_{\infty})$, and $c\in\mathbb{Z}[N^{S}]$.
Put
\[
X:=(\tilde{M}_{1}/\tilde{M}_{2})\times\prod_{\mathfrak{p}\in R_{2}\setminus R_{1}}(O_{\mathfrak{p}}^{\times}/V_{\mathfrak{p}}).
\]
From (\ref{eq:fRMW_Rchange}) and (\ref{eq:fRMW_Mchange}), we have
\[
f_{R_{1},M_{1},W\cap S_{f}}=\sum_{y\in X}yf_{R_{2},M_{2},W\cap S_{f}}.
\]
Thus it is enough to prove that $yxf_{R_{1},M_{1},W\cap S_{f}}\otimes b\otimes c\subset\mathcal{B}_{R_{2},M_{2}}(W)$
for $y\in X$. Since the natural map $\mathcal{E}_{R_{1},M_{1}}\to X$
is surjective, there exists lift $p\in\mathcal{E}_{R_{1},M_{1}}$
of $y\in X$. Then we have
\[
yxf_{R_{1},M_{1},W\cap S_{f}}\otimes b\otimes c:=\left(p^{-1}yxf_{R_{1},M_{1},W\cap S_{f}}\otimes p^{-1}b\otimes p^{-1}c\right)\otimes[p].
\]
Since $p^{-1}yxf_{R_{1},M_{1}W\cap S_{f}}\otimes p^{-1}b\otimes p^{-1}c\in\mathcal{B}_{R_{2},M_{2}}^{*}(W)$,
the lemma is proved.\end{proof}
\begin{defn}
For $R_{1}\subset R_{2}\subset S_{f}$ and $M_{2}\subset M_{1}\subset\mathbb{F}_{q}^{\times}$,
we define a natural transformation $i_{R_{2},M_{2}}^{R_{1},M_{1}}$
from $\mathcal{B}_{R_{1},M_{1}}\big|_{{\bf Sub}(S\setminus R_{2})}$
to $\mathcal{B}_{R_{2},M_{2}}$ by the natural inclusion map.
\end{defn}

\subsection{Definition and properties of $\vartheta_{R,M}\in H_{0}(F^{\times},\mathcal{B}_{R,M}(\emptyset))$}

Fix a subset $R$ of $S_{f}$ and a subgroup $M$ of $\mathbb{F}_{q}^{\times}$.
We put ${\rm Cl}_{R,M}:=N^{S\cup S_{q}}/\mathcal{E}_{R,M}$. Note
that ${\rm Cl}_{R,M}$ is a finite group.
\begin{defn}
For $c\in{\rm Cl}_{R,M}$, we define $\vartheta_{R,M,c}\in H_{0}(F^{\times},\mathcal{B}_{R,M}(\emptyset))$
as follows. Let $u\in N^{S\cup S_{q}}$ be a lift of $c$ and $z\in\mathcal{Z}(\mathcal{U},\emptyset)$
a lift of $\vartheta_{\infty}(\mathcal{U},F_{R,M}\cap\mathcal{O}_{F}^{\times})\in H_{0}(F_{R,M}\cap\mathcal{O}_{F}^{\times},\mathcal{Z}(\mathcal{U},\emptyset))$.
Let $\vartheta_{R,M,c}$ be the image of $uf_{R,M,\emptyset}\otimes z\otimes[u]\in\mathcal{B}_{R,M}(\emptyset)$.
This definition does not depend on the choice of $u$ and $z$.
\end{defn}

\begin{defn}
Define $\vartheta_{R,M}\in H_{0}(F^{\times},\mathcal{B}_{R,M}(\emptyset))$
by 
\[
\vartheta_{R,M}:=\sum_{c\in{\rm Cl}_{R,M}}\vartheta_{R,M,c}.
\]
\end{defn}
\begin{prop}
\label{prop:exist_lift}There exists a lift $\alpha\in H_{0}(F_{R,M}^{\times}\cap\mathcal{O}_{S}^{\times},\mathcal{B}_{R,M}(\emptyset))$
of $\vartheta_{R,M}$ such that ${\rm r}_{\{v\}}^{\emptyset}(\alpha)\in H_{0}(F_{R,M}^{\times}\cap\mathcal{O}_{S}^{\times},\mathcal{B}_{R,M}(\{v\}))$
vanishes for all $v\in S\setminus R$.\end{prop}
\begin{proof}
Put $E_{S}=F_{R,M}^{\times}\cap\mathcal{O}_{S}^{\times}$, $E:=F_{R,M}^{\times}\cap\mathcal{O}_{F}^{\times}$,
$X:=\prod_{v\in S_{f}\setminus R}(F_{v}^{\times}/O_{v}^{\times})$
and $Y:=\prod'_{v\notin S_{\infty}\cup R\cup\{\mathfrak{q}\}}(F_{v}^{\times}/O_{v}^{\times})$.
For $W\subset S_{f}\setminus R$ and $u\in N^{S\cup S_{q}}$, we put
\[
g_{W}(u):=uf_{R,M,W}\otimes z\otimes u\in H_{0}(E,\mathcal{B}_{R,M}^{*}(W))
\]
where $z$ is a lift of $ $$\vartheta_{\infty}(\mathcal{U},E)$.
For each $v\in S_{f}\setminus R$, fix an element $a(v)$ of 
\[
F_{R,M}^{\times}\cap(\pi_{v}\mathcal{O}_{v}^{\times})\cap\left(\bigcap_{\mathfrak{p}\in(S_{q}\cup S_{f})\setminus\{v\}}\mathcal{O}_{\mathfrak{p}}^{\times}\right).
\]
From (\ref{eq:fRMW_Wchange}) and Proposition \ref{prop:theta_inf_mul_p},
for $v\in S_{f}\setminus R$, we have
\begin{align}
{\rm r}_{\{v\}}^{\emptyset}(g_{\emptyset}(u)) & =g_{\{v\}}(w)\otimes[a(v)^{-1}]-g_{\{v\}}(u)\otimes[1]\label{eq:gv}
\end{align}
where $w\in N^{S\cup S_{q}}$ is the image of $a(v)u$. From Proposition
\ref{prop:theta_inf_vanish}, for $v\in S_{\infty}$, we have 
\begin{equation}
{\rm r}_{\{v\}}^{\emptyset}(g_{\emptyset}(u))=0\in H_{0}(E,\mathcal{B}_{R,M}(\{v\})).\label{eq:vanish_at_inf}
\end{equation}
For $u\in N^{S\cup S_{q}}$ and $C\in{\rm im}(X\to Y/F_{R,M}^{\times}),$
we define $h_{W}(C,u)\in H_{0}(E_{S},\mathcal{B}_{R,M}(W))$ as the
image of
\[
g_{W}(tu)\otimes[b]\in H_{0}(E,\mathcal{B}_{R,M}(W))
\]
where $t\in N^{S\cup S_{q}}\subset Y$ is a lift of $C$ and $b$
is an element of $F_{R,M}^{\times}$ such that $bt\in{\rm im}(X\to Y)$.
This definition does not depend on the choice of $t$ and $b$. From
(\ref{eq:gv}), for $v\in S_{f}\setminus R$, we have
\begin{align}
h_{\emptyset}(C,u) & =g_{\emptyset}(tu)\otimes[b]\nonumber \\
 & =g_{\{v\}}(a(v)tu)\otimes[ba(v)^{-1}]-g_{\{v\}}(tu)\otimes[b]\nonumber \\
 & =h_{\{v\}}(\pi_{v}C,u)-h_{\{v\}}(C,u).\label{eq:vanish_at_v}
\end{align}
For each $C\in{\rm coker}(X\to Y/F_{R,M}^{\times})$, fix a lift $u_{C}\in N^{S\cup S_{q}}$.
Put 
\begin{eqnarray*}
h(u) & = & \sum_{C\in{\rm im}(X\to Y/F_{R,M}^{\times})}h_{\emptyset}(C,u)\\
\alpha & = & \sum_{C\in{\rm coker}(X\to Y/F_{R,M}^{\times})}h(u_{C}).
\end{eqnarray*}
Then $\alpha$ is a lift of $\vartheta_{R,M}$. For all $v\in S\setminus R$,
we have ${\rm r}_{\{v\}}^{\emptyset}(\alpha)=0$ since ${\rm r}_{\{v\}}^{\emptyset}(h(u))=0$
from (\ref{eq:vanish_at_inf}) and (\ref{eq:vanish_at_v}). Thus the
claim is proved.
\end{proof}

\begin{prop}
\label{prop:vanish_of_vartheta}For all $v\in S\setminus R$, we have
\[
\vartheta_{R,M}\in\ker({\rm r}_{\{v\}}^{\emptyset}:H_{0}(F^{\times},\mathcal{B}_{R,M}(\emptyset))\to H_{0}(F^{\times},\mathcal{B}_{R,M}(\{v\}))).
\]
\end{prop}
\begin{proof}
The claim follows from Proposition \ref{prop:exist_lift}.\end{proof}
\begin{prop}
\label{prop:i_rmrm_theta}We have $i_{R_{2},M_{2}}^{R_{1},M_{1}}(\vartheta{}_{R_{1},M_{1}})=\vartheta{}_{R_{2},M_{2}}$.\end{prop}
\begin{proof}
Let $j:F^{\times}\to N^{S\cup S_{q}}$ be a natural projection. Put
$E_{i}=\mathcal{E}_{R_{i},M_{i}}$, $E_{i}'=E_{i}\cap\mathcal{O}_{F}^{\times}$,
and $E_{i}''=j(E_{i})$ for $i=1,2$. Then the natural sequence
\[
1\to E_{1}'/E_{2}'\to E_{1}/E_{2}\xrightarrow{j}E_{1}''/E_{2}''
\]
is exact. For $u\in N^{S\cup S_{q}}/E_{2}''$ and $z\in H_{0}(E_{2},\mathcal{Z}(\mathcal{U},\emptyset))$,
we put
\[
a(u,z):=uf_{R_{2},M_{2}}\otimes z\otimes[u]\in H_{0}(F^{\times},\mathcal{B}_{R_{2},M_{2}}(\emptyset)).
\]
Fix a lift $y\in H_{0}(E_{2},\mathcal{Z}(\mathcal{U},\emptyset))$
of $ $$\vartheta_{\infty}(\mathcal{U},E_{1})$. Since
\[
f_{R_{1},M_{1}}=\sum_{g\in E_{1}/E_{2}}g^{-1}j(g)f_{R_{2},M_{2}},
\]
for $u\in N^{S\cup S_{q}}$ we have 
\begin{align*}
uf_{R_{1},M_{1}}\otimes y\otimes[u] & =\sum_{g\in E_{1}/E_{2}}\left(j(g)uf_{R_{2},M_{2}}\otimes gy\otimes[gu]\right)\otimes[g^{-1}]\\
 & \equiv\sum_{g\in E_{1}/E_{2}}a(j(g)u,gy)\\
 & \equiv\sum_{g'\in E_{1}''/E_{2}''}\sum_{\substack{g\in E_{1}/E_{2}\\
j(g)=g'
}
}a(g'u,gy)\\
 & \equiv\sum_{g'\in E_{1}''/E_{2}''}a(g'u,\vartheta_{\infty}(\mathcal{U},E_{2}))\pmod{I_{F^{\times}}\mathcal{B}_{R_{2},M_{2}}(\emptyset)}.
\end{align*}
Thus we have
\begin{align*}
i_{R_{2},M_{2}}^{R_{1},M_{1}}(\vartheta{}_{R_{1},M_{1}}) & \equiv\sum_{u\in N^{S\cup S_{q}}/E_{1}''}uf_{R_{1},M_{1}}\otimes y\otimes[u]\\
 & \equiv\sum_{u\in N^{S\cup S_{q}}/E_{2}''}a(u,\vartheta_{\infty}(\mathcal{U},E_{2}))\\
 & \equiv\vartheta_{R_{2},M_{2}}\pmod{I_{F^{\times}}\mathcal{B}_{R_{2},M_{2}}(\emptyset)}.
\end{align*}
Thus the claim is proved.
\end{proof}

\subsection{Evaluation map}

We denote by $\mathbb{Q}\otimes\mathcal{R}$ the functor from ${\bf Sub}(S)$
to ${\bf Mod}(\mathbb{Z}[F^{\times}])$ defined by $(\mathbb{Q}\otimes\mathcal{R})(W)=\mathbb{Q}\otimes(\mathcal{R}(W))$.
Fix a subset $R$ of $S_{f}$ and a subgroup $M$ of $\mathbb{F}_{q}^{\times}$.
Let $W$ be a subset of $S\setminus R$. In this section, we define
a natural transformation $\mathcal{L}_{R,M}|_{{\bf Sub}(S\setminus R)\setminus\{S\}}$
from $\mathcal{B}_{R,M}$ to $(\mathbb{Q}\otimes\mathcal{R})|_{{\bf Sub}(S\setminus R)\setminus\{S\}}$.
Put $W_{\infty}:=W\cap S_{\infty}$ and $W_{f}:=W\cap S_{f}$. For
$a\in\mathcal{A}_{R,M}(W_{f})$ and $e\in N_{S_{f}\setminus W_{f}}$,
define $\Phi_{a,e}\in\mathcal{S}(\mathbb{A}_{F}^{S_{\infty}})$ by
\[
\Phi_{a,e}(x):=\begin{cases}
a(x) & y\in j^{-1}(e)\\
0 & y\notin j^{-1}(e)
\end{cases}
\]
where $j:\mathbb{A}_{F}^{S_{\infty}}\to N_{S_{f}\setminus W_{f}}$
is the projection map.
\begin{lem}
\label{lem:inv}For all $\mathfrak{p}\in S_{q}\setminus\{\mathfrak{q}\}$
and $u\in\mathcal{U}$, $\Phi_{a,e}$ is invariant under the action
of $uO_{\mathfrak{p}}$.\end{lem}
\begin{proof}
Assume that $\mathfrak{p}\notin S_{f}\setminus W$. Then $a\in\mathcal{S}(\mathbb{A}_{F}^{S_{\infty}})$
is invariant under the action of $\mathfrak{p}O_{\mathfrak{p}}$,
and $j^{-1}(e)$ is invariant under the action of $F_{\mathfrak{p}}$.
Thus $\Phi_{a,e}$ is invariant under the action of $\mathfrak{p}O_{\mathfrak{p}}\supset uO_{\mathfrak{p}}$.
Assume that $\mathfrak{p}\in S_{f}\setminus W$. Then $a\in\mathcal{S}(\mathbb{A}_{F}^{S_{\infty}})$
is invariant under the action of $\mathfrak{p}O_{\mathfrak{p}}$.
If ${\rm ord}_{\mathfrak{p}}(x)\neq0$ then $\Phi_{a,e}=0$. If ${\rm ord}_{\mathfrak{p}}(e)\neq0$
then $j^{-1}(e)$ is invariant under the action of $uO_{\mathfrak{p}}$.
Thus $\Phi_{a,e}$ is invariant under the action of $\mathfrak{p}O_{\mathfrak{p}}\cap uO_{\mathfrak{p}}=uO_{\mathfrak{p}}$.
Thus the lemma is proved.\end{proof}
\begin{lem}
$\Phi_{a,e}$ is $\mathcal{U}$-smooth for all $a$ and $e$.\end{lem}
\begin{proof}
Let $u\in\mathcal{U}$ and $x\in\mathbb{A}_{F}^{S_{\infty}}$. It
is enough to prove that
\[
\int_{\mathbb{Z}_{q}}\Phi_{a,e}(x+tu)dt=0.
\]
From Lemma \ref{lem:inv}, we have
\[
\int_{\mathbb{Z}_{q}}\Phi_{a,e}(x+tu)dt=\int_{\mathbb{Z}_{q}}\Phi_{a,e}(x+tu_{\mathfrak{q}})dt.
\]
We denote by $A$ the subgroup of $\mathcal{S}(F_{\mathfrak{q}})$
generated by $\{h_{x}\mid x\in\mathbb{F}_{q}\subset\kappa_{\mathfrak{q}}\}$
where $h_{x}=\bm{1}_{\pi^{-1}(x)}-\bm{1}_{\pi^{-1}(0)}$ and $\pi:O_{\mathfrak{q}}\to O_{\mathfrak{q}}/\mathfrak{q}O_{\mathfrak{q}}=\kappa_{\mathfrak{q}}$
is a natural projection. From the definition, $\Phi_{a,e}$ is in
$\mathcal{S}(\mathbb{A}_{F}^{S_{\infty}\cup\{\mathfrak{q}\}})\otimes A$.
Thus it is enough to prove that
\[
\int_{\mathbb{Z}_{q}}\Phi(x'+tu_{\mathfrak{p}})dt=0
\]
for all $\Phi\in A$ and $x'\in F_{\mathfrak{q}}$. This follows from
the fact that $u_{\mathfrak{p}}\in\pi^{-1}(\mathbb{F}_{q}^{\times})$.\end{proof}
\begin{defn}
We define a natural transformation $\mathcal{L}_{R,M}$ from $\mathcal{B}_{R,M}|_{{\bf Sub}(S\setminus R)\setminus\{S\}}$
to $(\mathbb{Q}\otimes\mathcal{R})|_{{\bf Sub}(S\setminus R)\setminus\{S\}}$
as follows. Fix $W\subset S\setminus R$ such that $W\neq S$.

(The case $W_{\infty}\neq S_{\infty}$ ). For $a\in\mathcal{A}_{R,M}(W_{f})$,
$b=(b_{g})_{g\in N_{S\infty\setminus W_{\infty}}}\in\mathcal{Z}(\mathcal{U},W_{\infty})$,
$g\in N_{S_{\infty}\setminus W_{\infty}}$ and $x\in N_{S_{f}\setminus W_{f}}$,
we put
\[
\zeta_{a,b,g,x}(\bm{s})=\sum_{g\in N_{S_{\infty}\setminus W_{\infty}}}\sum_{x\in N_{S_{f}\setminus W_{f}}}\sum_{y\in F^{\times}}\mathcal{L}_{\infty}(b_{g})(y)\cdot\Phi_{a,x}(y)\cdot\prod_{v\in S\setminus W}|y|_{v}^{-s_{v}}\ \ \ \ \ (\bm{s}=(s_{v})_{v\in S\setminus R}).
\]
Then $\zeta_{a,b,g,x}(\bm{s})$ is analytically continued to the whole
$\mathbb{C}^{\#(S\setminus W)}$. Let $Z_{a,b}\in\mathbb{R}[(s_{v})_{v\in S\setminus W}]$
be a taylor series at $\bm{s}=0$. We put

\[
\mathcal{L}_{R,M}(W)(a\otimes b\otimes c)=c\sum_{g\in N_{S_{\infty}\setminus W_{\infty}}}\sum_{x\in N_{S_{f}\setminus W_{f}}}Z_{a,b,g,x}\otimes[gx].
\]

(The case $W_{\infty}=S_{\infty}$). We put 
\[
\mathcal{L}_{R,M}(W)(a\otimes b\otimes c):=c\sum_{x\in N_{S_{f}\setminus W_{f}}}\langle\!\langle\widetilde{\mathcal{L}_{\infty}(b)},\Phi_{a,x}\rangle\!\rangle\prod_{v\in S_{f}\setminus W_{f}}\left|x\right|_{v}^{-s_{v}}\otimes[x]\ \ \ \ 
\]
where $a\in\mathcal{A}_{R,M}(W\cap S_{f})$, $b\in Z(\mathcal{U})=\mathcal{Z}(\mathcal{U},S_{\infty})$,
$c\in\mathbb{Z}[N^{S}]$ and $\widetilde{\mathcal{L}_{\infty}(b)}\in\mathcal{K}'(\mathcal{U})$
is a lift of $\mathcal{L}_{\infty}(b)\in\mathcal{K}(\mathcal{U})$.
This definition does not depend on the choice of $\widetilde{\mathcal{L}_{\infty}(b)}$
because $\langle\!\langle\bm{1}_{F^{\times}},\Phi_{a,x}\rangle\!\rangle=0$
from $W_{f}\neq S_{f}$ and $\Phi_{a,x}(0)=0$.
\end{defn}
From the definition, we have 
\begin{equation}
\mathcal{L}_{R_{2},M_{2}}\circ i_{R_{2},M_{2}}^{R_{1},M_{1}}=\mathcal{L}_{R_{1},M_{1}}.\label{eq:i_rmrm_L}
\end{equation}

\begin{prop}
\label{prop:L_and_theta}We have $\lambda(\mathcal{L}_{R,M}(\vartheta_{R,M}))=\Theta_{S,\{\mathfrak{q}\},K}(s)$.\end{prop}
\begin{proof}
We write $J$ for the group of fractional ideals of $\mathcal{O}_{S}$.
Let $\lambda$ be a natural isomorphism between $N^{S}$ and $J$.
For $\mathfrak{a}\in J$, we write $\sigma_{\mathfrak{a}}:={\rm rec}(\lambda^{-1}(\mathfrak{a}))\in{\rm Gal}(K/F)$.
We put $A:=\{(x,\mathfrak{a})\mid\mathfrak{a}\in J,x\in F^{\times}\}=J\times F^{\times}$.
We define the action of $F^{\times}$ to $A$ by $p(x,\mathfrak{a})=(px,p\mathfrak{a})$.
We define ${\rm ev}\in{\rm Hom}_{\mathbb{Z}[F^{\times}]}(\mathcal{B}_{R,M}(\emptyset)\to{\rm Map}(A,\mathbb{Z}))$
by
\[
({\rm ev}(a\otimes b\otimes c))(x,\mathfrak{a}):=\begin{cases}
a(x)\times\mathcal{L}_{\infty}(b)(x) & \mathfrak{a}=\lambda(c)\\
0 & \mathfrak{a}\neq\lambda(c),
\end{cases}
\]
where $a\in\mathcal{S}(A_{F}^{S_{\infty}})$, $b\in\mathcal{Z}(F^{\times},\emptyset)$,
$c\in\mathbb{Z}[N^{S}]$, and $(x,\mathfrak{a})\in A$. Then we have
\[
\mathcal{L}(h)=\sum_{(x,\mathfrak{a})\in A}{\rm ev}(h)(x,\mathfrak{a})\left(\prod_{v\in S}\left|x\right|_{v}^{-s_{v}}\right)\otimes[\lambda^{-1}(\mathfrak{a})\prod_{v\in S}x_{v}].
\]
and
\[
\lambda(\mathcal{L}(h))=\sum_{(x,\mathfrak{a})\in A}{\rm ev}(h)(x,\mathfrak{a})\left(\prod_{v\in S}\left|x\right|_{v}^{-s}\right)N(\mathfrak{a})^{s}\otimes[\sigma_{x^{-1}\mathfrak{a}}].
\]
Define $D\in{\rm Map}(A,\mathbb{Z})$ by
\[
D((x,\mathfrak{a}))=\begin{cases}
1 & x\in\mathfrak{a}\setminus\mathfrak{aq}\\
N(\mathfrak{q})-1 & x\in\mathfrak{aq}\\
0 & x\notin\mathfrak{a},
\end{cases}
\]
where $\mathfrak{a}$ and $\mathfrak{ap}$ are fractional ideals of
$\mathcal{O}_{S}$. Let $t\in\mathcal{B}_{R,M}(\emptyset)$ be a lift
of $\vartheta$. Then from Proposition \ref{prop:D_is_signed_fd},
we have
\[
\sum_{p\in F^{\times}}[p]{\rm ev}(t)=D.
\]
Thus we have
\begin{align*}
\lambda(\mathcal{L}(t)) & =\sum_{(x,\mathfrak{a})\in A/F^{\times}}D(x,\mathfrak{a})\left(\prod_{v\in S}\left|x\right|_{v}^{-s}\right)N(\mathfrak{a})^{s}\otimes[\sigma_{x^{-1}\mathfrak{a}}]\\
 & =\sum_{\mathfrak{a}\in J}D(1,\mathfrak{a})N(\mathfrak{a})^{s}\otimes[\sigma_{\mathfrak{a}}]\\
 & =\Theta_{S,\{\mathfrak{q}\},K}(s).
\end{align*}

\end{proof}

\subsection{Definition of ${\rm Sh}^{v_{0}}$, ${\rm Sh}^{\diamond}$ and ${\rm Sh}^{v_{0},\diamond}$}
\begin{lem}
Let $m$ be a positive integer defined by
\[
m:=\begin{cases}
1 & M=\mathbb{F}_{q}^{\times}\ {\rm and}\ {\rm ch}(\mathfrak{q})\geq[F:\mathbb{Q}]+2\\
{\rm ch}(\mathfrak{q})^{[F:\mathbb{Q}]} & {\rm otherwise}.
\end{cases}
\]
Let $V$ be a subset of $S\setminus R$. If $M=1$, $S_{\infty}\cap V\neq S_{\infty}$
or $R\neq\emptyset$ then ${\rm Sh}_{R,M,V}:=(\mathcal{B}_{R,M}|_{{\bf Sub}(V)\setminus\{S\}},m\mathcal{L}_{R,M}|_{{\bf Sub}(V)\setminus\{S\}},\vartheta_{R,M},m)$
is a Shintani datum on $V$ for $(\mathcal{R},\Upsilon,\lambda,\theta)$.\end{lem}
\begin{proof}
First, we need to prove that
\begin{equation}
m\mathcal{L}_{R,M}(W)(\mathcal{B}_{R,M}(W))\subset\mathcal{R}(W)\ \ \ (W\subset V,\ W\neq S).\label{eq:f1}
\end{equation}
Let $a\in\mathcal{A}_{R,M}(W\cap S_{f})$, $b=(b_{g})_{g}\in\prod_{g\in N_{S_{\infty}\setminus W_{\infty}}}Z(\mathcal{U}_{g})$
and $c\in\mathbb{Z}[N^{S}]$. Since the constant term of $m\mathcal{L}_{R,M}(W)(a\otimes b\otimes c)$
is given by
\[
c\sum_{g\in N_{S_{\infty}\setminus W_{\infty}}}\sum_{x\in N_{S_{f}\setminus W_{f}}}m\langle\!\langle\mathcal{L}_{\infty}(b_{g}),\Phi_{a,x}\rangle\!\rangle[x],
\]
it is enough to prove that
\begin{equation}
m\langle\!\langle\mathcal{L}_{\infty}(b_{g}),\Phi_{a,x}\rangle\!\rangle\in\mathbb{Z}.\label{eq:f2}
\end{equation}
The claim (\ref{eq:f2}) follows from Lemma \ref{lem:strongIntegrality}
($m=1$) and Lemma \ref{lem:weakIntegrality} ($m\neq1$). Hence (\ref{eq:f1})
is proved. Therefore $m\mathcal{L}_{R,M}:\mathcal{B}_{R,M}\to\mathcal{R}$
is well-defined. The other conditions are followed from Proposition
\ref{prop:homology_vanish}, \ref{prop:vanish_of_vartheta} and \ref{prop:L_and_theta}.\end{proof}
\begin{defn}
For $\mathfrak{q}\notin S$ such that ${\rm ch}(\mathfrak{q})\geq[F:\mathbb{Q}]+2$,
we define Shintani data ${\rm Sh}^{v_{0}}$ and ${\rm Sh}^{v_{0},\diamond}$
on $S\setminus\{v_{0}\}$ by 
\begin{eqnarray*}
{\rm Sh}^{v_{0}} & = & {\rm Sh}_{S_{f}\cap\{v_{0}\},\mathbb{F}_{q}^{\times},S\setminus\{v_{0}\}}\\
{\rm Sh}^{v_{0},\diamond} & = & {\rm Sh}_{S_{f}\cap\{v_{0}\},\{1\},S\setminus\{v_{0}\}.}
\end{eqnarray*}
For $\mathfrak{q}\notin S$ such that ${\rm ch}(\mathfrak{q})\neq2$,
we define a Shintani datum ${\rm Sh}^{\diamond}$ on $S$ by
\[
{\rm Sh}^{\diamond}={\rm Sh}_{\emptyset,1,S}.
\]

\end{defn}
From Proposition \ref{prop:i_rmrm_theta} and relation (\ref{eq:i_rmrm_L}),
the natural transformations $i_{S_{f}\cap\{v_{0}\},1}^{S_{f}\cap\{v_{0}\},\mathbb{F}_{q}^{\times}}$
and $i_{S_{f}\cap\{v_{0}\},1}^{\emptyset,1}$ give morphisms ${\rm Sh}^{v_{0}}\to{\rm Sh}^{v_{0},\diamond}$
and ${\rm Sh}^{\diamond}\to{\rm Sh}^{v_{0},\diamond}$ respectively.

\section{Proof of the main theorems\label{sec:proof}}

We put
\begin{align*}
X & :=\{\mathfrak{q}\notin S,{\rm ch}(\mathfrak{q})\geq[F:\mathbb{Q}]+2\}\\
Y & :=\{\mathfrak{q}\notin S,{\rm ch}(\mathfrak{q})\neq2\}.
\end{align*}
In Section \ref{sec:constSh}, we constructed a Shintani datum ${\rm Sh}^{v_{0}}$
for $\mathfrak{q}\in X$ and Shintani data ${\rm Sh}^{\diamond}$,
${\rm Sh}^{v_{0},\diamond}$ for $\mathfrak{q}\in Y$. To avoid confusion,
we write ${\rm Sh}_{(\mathfrak{q})}^{v_{0}}$, ${\rm Sh}_{(\mathfrak{q})}^{\diamond}$
and ${\rm Sh}_{(\mathfrak{q})}^{v_{0},\diamond}$ for these Shintani
data. For $v\in S$, put $J_{v}:=\ker(\mathcal{R}(\emptyset)\to\mathcal{R}(\{v\}))$.
\begin{defn}
For a Shintani datum on $W\subsetneq S$, we define 
\[
Q^{N}({\rm Sh})\in(\prod_{v\in W}I_{v})/(I_{F^{\times}}\prod_{v\in W}I_{v})
\]
 as follows. The homomorphism
\[
\mathcal{R}(\emptyset)\to\mathbb{Z}[N_{F}]\ ;\ P(s_{v_{0}},\dots,s_{v_{r}})\otimes[c]\mapsto P(0,\dots,0)[c]
\]
naturally induces the homomorphism
\[
(\prod_{v\in W}J_{v})/(I_{F^{\times}}\prod_{v\in W}J_{v})\to(\prod_{v\in W}I_{v})/(I_{F^{\times}}\prod_{v\in W}I_{v}).
\]
Then $Q^{N}({\rm Sh})$ the image of $Q({\rm Sh})$ under this homomorphism.
\end{defn}
In this section, we put
\[
V:=S\setminus\{v_{0}\}.
\]

\begin{defn}
We define $\hat{\Theta}_{S,\mathfrak{q},K}$ by
\[
\hat{\Theta}_{S,\mathfrak{q},K}:=Q^{N}({\rm Sh}_{(\mathfrak{q})}^{v_{0}}).
\]
\end{defn}
\begin{lem}
\label{lem:equiv_conditions_T}Let $T$ be a finite set of places
of $F$ such that $S\cap T=\emptyset$. The conditions 
\[
\ker(\mu_{K}\to\prod_{\mathfrak{p}\in T_{_{K}}}(O_{K}/\mathfrak{p})^{\times})=\{1\}
\]
and
\[
\delta_{T}\in{\rm Ann}_{\mathbb{Z}[G]}(\mu_{K})
\]
are equivalent.\end{lem}
\begin{proof}
Put $m=\#\mu_{K}$. The both conditions are equivalent to the condition
\[
\forall p\mid m,\exists\mathfrak{p}\in T,{\rm ch}(\mathfrak{p})\neq p.
\]
\end{proof}
\begin{lem}
\label{lem:JAnn}Let $J\subset\mathbb{Z}[G]$ be the ideal spanned
by 
\[
\{1-N(\mathfrak{q})\sigma_{\mathfrak{q}}^{-1}\mid\mathfrak{q}\in X\}.
\]
Then we have $J={\rm Ann}_{\mathbb{Z}[G]}(\mu_{K})$.\end{lem}
\begin{proof}
Put $m=\#\mu_{K}$ and $n=\gcd\{1-N(\mathfrak{q})\mid\mathfrak{q}\notin X,\sigma_{\mathfrak{q}}|_{K}=1\}$.
Let us prove that $m=n$.$ $ It is obvious that $m\mid n$. Assume
that $n>m$. Then $K(\zeta_{n})/K$ is not a trivial extension. Let
$\sigma\in{\rm Gal}(K(\zeta_{n})/K)\subset{\rm Gal}(K(\zeta_{n})/F)$
be a non trivial element. From Chevotarev's density theorem, there
exists $\mathfrak{q}\in X$ such that
\[
\sigma_{\mathfrak{q}}|_{{\rm Gal}(K(\zeta_{n})/F)}=\sigma.
\]
Then $\sigma_{\mathfrak{q}}|_{K}=\sigma|_{K}={\rm id}$. From the
norm functoriality of the reciprocity map, we have
\[
\sigma_{N(\mathfrak{q})}\neq1\in{\rm Gal}(\mathbb{Q}(\zeta_{n})/\mathbb{Q}).
\]
Therefore, $n\nmid1-N(\mathfrak{q})$. This contradicts to the assumption
that $n=\gcd\{1-N(\mathfrak{q})\mid\mathfrak{q}\in X,\sigma_{\mathfrak{q}}=1\}$.
Therefore the assumption $n>m$ must be false. Thus we have $n=m$
and $m\in J$. From Chevotarev's density theorem, for all $\sigma\in G$,
there exists $\mathfrak{q}\in X$ such that $\sigma_{\mathfrak{q}}=\sigma$.
Thus we have $J={\rm Ann}_{\mathbb{Z}[G]}(\mu_{K})$.\end{proof}
\begin{lem}
\label{lem:vanishq}For $\mathfrak{q}\in X$, we have
\[
(1-N(\mathfrak{q})\sigma_{\mathfrak{q}}^{-1})\Theta_{S,K}\in\prod_{v\in V}I_{G_{v}}.
\]
\end{lem}
\begin{proof}
It follows from
\[
{\rm rec}(\hat{\Theta}_{S,\mathfrak{q},K})=(1-N(\mathfrak{q})\sigma_{\mathfrak{q}}^{-1})\Theta_{S,K}
\]
and
\[
\hat{\Theta}_{S,\mathfrak{q},K}\in\prod_{v\in V}I_{v}.
\]
\end{proof}
\begin{lem}
[Vanishing order part of Gross conjecture]\label{lem:vanishT}Let
$T$ be a finite set of places of $F$ such that $S\cap T=\emptyset$.
If $\ker(\mu_{K}\to\prod_{\mathfrak{p}\in T_{_{K}}}(O_{K}/\mathfrak{p})^{\times})=\{1\}$
then we have
\[
\Theta_{S,T,K}\in\prod_{v\in V}I_{G_{v}}.
\]
\end{lem}
\begin{proof}
From Lemma \ref{lem:JAnn} and \ref{lem:vanishq}, we have
\[
J\Theta_{S,K}\subset\prod_{v\in V}I_{G_{v}}.
\]
Since $\delta_{T}\in J$ from Lemma \ref{lem:equiv_conditions_T},
we have
\[
\Theta_{S,T,K}=\delta_{T}\Theta_{S,K}\in\prod_{v\in V}I_{G_{v}}.
\]
\end{proof}
\begin{lem}
\label{lem:2multiply}We have 
\begin{eqnarray*}
2\prod_{v\in V}I_{v} & \subset & (\sum_{v\in V}I_{v})\prod_{v\in V}I_{v}\\
2\prod_{v\in V}I_{G_{v}} & \subset & (\sum_{v\in V}I_{G_{v}})\prod_{v\in V}I_{G_{v}}.
\end{eqnarray*}
\end{lem}
\begin{proof}
Since $F\neq\mathbb{Q}$, $S_{\infty}\setminus\{v_{0}\}\neq\emptyset$.
Let $v'$ be any place in $S_{\infty}\setminus\{v_{0}\}$. Then we
have $2I_{v'}\subset I_{v'}^{2}$ (resp. $2I_{G_{v'}}\subset I_{G_{v'}}^{2}$)
since 
\[
2([1]-[x])=([1]-[x])^{2}
\]
 for all $x\in N_{v'}$ (resp. $x\in G_{v'}$). Thus the lemma is
proved.
\end{proof}

\begin{lem}
\label{lem:vanish_2T}Let $T$ be a finite set of places of $F$ such
that $S\cap T=\emptyset$. If $\ker(\mu_{K}\to\prod_{\mathfrak{p}\in T_{_{K}}}(O_{K}/\mathfrak{p})^{\times})=\{1\}$
then we have
\[
2\Theta_{S,T,K}\in I_{G_{H}}\prod_{v\in V}I_{G_{v}}.
\]
\end{lem}
\begin{proof}
The claim follows from Lemma \ref{lem:vanishT} and Lemma \ref{lem:2multiply}.
\end{proof}
Recall that $\mathcal{R}(\emptyset)=\mathcal{P}_{\emptyset}\otimes\mathbb{Z}[N_{F}]$.
For $v\in S$, define $g_{v}:F^{\times}\to\mathcal{R}(\emptyset)^{\times}$
by
\[
g_{v}(x)=\left|x\right|_{v}^{-s_{v}}\otimes[f_{v}(x)].
\]
For $v\in S$, put $J_{v}:=\ker(\mathcal{R}(\emptyset)\to\mathcal{R}(\{v\}))$.
\begin{lem}
\label{lem:detred}For $x_{1},\dots,x_{r}\in F^{\times}$, we have
\[
{\rm det}(-1+\prod_{s=j}^{r}g_{v_{s}}(x_{i}))_{i,j=1}^{r}\equiv{\rm det}(-1+g_{v_{j}}(x_{i}))_{i,j=1}^{r}\pmod{(J_{v_{1}}+\cdots+J_{v_{r}})J_{v_{1}}\cdots J_{v_{r}}}.
\]
\end{lem}
\begin{proof}
For $k\in\{0,1,\dots,r\}$ and $y_{1},\dots,y_{k}\in F^{\times}$,
we put
\begin{align*}
a_{k}(y_{1},\dots,y_{k}) & :=\det(-1+\prod_{s=j}^{k}g_{v_{s}}(y_{i}))_{i,j=1}^{k}\\
b_{k}(y_{1},\dots,y_{k}) & :=\det(-1+g_{v_{j}}(y_{i}))_{i,j=1}^{k}.
\end{align*}
Let ${\rm P}(k)$ be the following statement:

For any $y_{1},\dots,y_{k}\in F^{\times}$, we have
\[
a_{k}(y_{1},\dots,y_{k})\equiv b_{k}(y_{1},\dots,y_{k})\pmod{(J_{v_{1}}+\cdots+J_{v_{k}})J_{v_{1}}\cdots J_{v_{k}}}.
\]

Note that ${\rm P(k)}$ implies $a_{k}(y_{1},\dots,y_{k})\in J_{v_{1}}\cdots J_{v_{k}}$
since $b_{k}(y_{1},\dots,y_{k})\in J_{v_{1}}\cdots J_{v_{k}}$. We
prove ${\rm P}(k)$ for $k=0,\dots,r$ by the induction on $k$. Assume
that ${\rm P}(k-1)$ holds. We have
\[
a_{k}(y_{1},\dots,y_{k})=\sum_{i=1}^{k}(-1)^{i-1}a_{k-1}(y_{1},\dots,\hat{y}_{i},\dots,y_{k})(-1+g_{v_{k}}(y_{i}))\prod_{i'\neq i}g_{v_{k}}(y_{i'}).
\]
Since $a_{k-1}(y_{1},\dots,\hat{y}_{i},\dots,y_{k})\in J_{v_{1}}\cdots J_{v_{k-1}}$
and $1-\prod_{i'\neq i}g_{v_{k}}(y_{i'})\in J_{v_{k}}$, we have
\begin{equation}
a_{k}(y_{1},\dots,y_{k})\equiv\sum_{i=1}^{k}(-1)^{i-1}a_{k-1}(y_{1},\dots,\hat{y}_{i},\dots,y_{k})(-1+g_{v_{k}}(y_{i}))\pmod{J_{v_{1}}\cdots J_{v_{k-1}}J_{v_{k}}^{2}}.\label{eq:epc1}
\end{equation}
Since $-1+g_{v_{k}}(y_{i})\in J_{v_{k}}$ and 
\[
a_{k-1}(y_{1},\dots,\hat{y}_{i},\dots,y_{k})\equiv b_{k-1}(y_{1},\dots,\hat{y}_{i},\dots,y_{k})\pmod{(J_{v_{1}}+\cdots J_{v_{k-1}})J_{v_{1}}\cdots J_{v_{k-1}}},
\]
we have
\begin{multline}
a_{k-1}(y_{1},\dots,\hat{y}_{i},\dots,y_{k})(-1+g_{v_{k}}(y_{i}))\equiv b_{k-1}(y_{1},\dots,\hat{y}_{i},\dots,y_{k})(-1+g_{v_{k}}(y_{i}))\\
\pmod{(J_{v_{1}}+\cdots J_{v_{k-1}})J_{v_{1}}\cdots J_{v_{k}}}.\label{eq:epc2}
\end{multline}
From (\ref{eq:epc1}) and (\ref{eq:epc2}), we have
\begin{align*}
a_{k}(y_{1},\dots,y_{k}) & \equiv\sum_{i=1}^{k}(-1)^{i-1}b_{k-1}(y_{1},\dots,\hat{y}_{i},\dots,y_{k})(-1+g_{v_{k}}(y_{i}))\\
 & \equiv b_{k}(y_{1},\dots,y_{k})\pmod{(J_{v_{1}}+\cdots+J_{v_{k}})J_{v_{1}}\cdots J_{v_{k}}}.
\end{align*}
Hence ${\rm P}(k)$ holds. Thus ${\rm P}(0),\dots,{\rm P}(r)$ are
proved by the induction. The lemma is equivalent to ${\rm P}(r)$.\end{proof}
\begin{lem}
\label{lem:classical_cnf}Let $T$ be a finite set of places of $F$
such that $S\cap T=\emptyset$ and $Y\cap T\neq\emptyset$. Let $\left\langle u_{1},\dots,u_{r}\right\rangle $
be a $\mathbb{Z}$-basis of $\mathcal{O}_{S,T}^{\times}$ such that
$(-1)^{\#T}\det(-\log\left|u_{i}\right|_{v_{j}})_{1\leq i,j\leq r}>0$.
Then we have
\[
\lim_{s\to0}s^{-r}\Theta_{S,T,H}(s)=n_{S,T}\det(-\log\left|u_{j}\right|_{v_{i}})_{i,j=1}^{r}\sum_{\sigma\in{\rm Gal}(H/F)}[\sigma].
\]
\end{lem}
\begin{proof}
The claim follows from the functional equation of $\Theta_{S,T,H}(s)$.
\end{proof}

\begin{lem}
\label{lem:calc_QNSH_reg}For $\mathfrak{q}\in Y$, we have
\[
Q^{N}({\rm Sh}^{\diamond}|_{V})\equiv{\rm ch}(\mathfrak{q})^{[F:\mathbb{Q}]}\hat{R}_{\mathfrak{q}}\pmod{I_{H}I_{v_{1}}\cdots I_{v_{r}}}.
\]
\end{lem}
\begin{proof}
Fix $\mathfrak{q}\in Y$ and put ${\rm Sh}^{\diamond}:=(\mathcal{B}^{\diamond},\mathcal{L}^{\diamond},\vartheta^{\diamond},{\rm ch}(\mathfrak{q})^{[F:\mathbb{Q}]})$.
From Lemma \ref{prop:calc_of_QSH}, we have
\[
Q({\rm Sh}^{\diamond}|_{V})=\eta_{2}(\eta_{1}(\vartheta^{\diamond})).
\]
Put $E=\mathcal{O}_{S,\mathfrak{\{q}\}}^{\times}$. Take a $\mathbb{Z}$-basis
$\left\langle u_{1},\dots,u_{r}\right\rangle $ of $E$ such that
$-\det(-\log\left|u_{i}\right|_{v_{j}})_{1\leq i,j\leq r}>0$. Note
that $\mathcal{R}(S)=\mathbb{Z}[N^{S}]$. From Lemma \ref{lem:eta1_is_in},
Proposition \ref{prop:homology_vanish} and Proposition \ref{prop:exist_lift},
we have
\[
\eta_{1}(\vartheta^{\diamond})\in{\rm im}(H_{r}(E,\mathbb{Z}[N^{S}])\to H_{r}(F^{\times},\mathbb{Z}[N^{S}])).
\]
Since the action of $E$ to $N^{S}$ is trivial, we have
\[
H_{r}(E,\mathbb{Z}[N^{S}])\simeq H_{r}(E,\mathbb{Z})\otimes\mathbb{Z}[N^{S}].
\]
Therefore there exists $A\in\mathbb{Z}[N^{S}]$ such that $\eta_{1}(\vartheta^{\diamond})$
is the image of 
\[
A\otimes\sum_{\sigma\in S_{r}}{\rm sgn}(\sigma)[1,u_{\sigma(1)},u_{\sigma(1)}u_{\sigma(2)},\dots,u_{\sigma(1)}\cdots u_{\sigma(r)}]\in\mathbb{Z}[N^{S}]\otimes_{\mathbb{Z}[F^{\times}]}\mathcal{I}_{r}.
\]
Let $\omega:H_{r}(F^{\times},N^{S})\simeq H(\mathcal{R}^{(S)}[-1]\to\mathcal{R}^{(S)}[0]\to\mathcal{R}^{(S)}[1])$
be a natural isomorphism. Then we have
\[
\omega(\eta_{1}(\vartheta^{\diamond}))=y
\]
where
\[
y=(y_{0},\dots,y_{r})\in\prod_{i=0}^{r}\mathcal{R}_{i,i}^{(S)}=\mathcal{R}^{(S)}[0]
\]
and 
\begin{align*}
y_{k} & =A\sum_{\sigma\in\mathfrak{S}_{r}}{\rm sgn}(\sigma)\prod_{i=k+1}^{r}\left(-1+\prod_{s=i}^{r}g_{v_{s}}(x_{\sigma(i)})\right)\otimes[1,x_{\sigma(1)},x_{\sigma(1)}x_{\sigma(2)},\dots,x_{\sigma(1)}\cdots x_{\sigma(k)}].\\
 & \in\mathcal{R}(\{v_{0},\dots,v_{k-1}\})\otimes\mathcal{I}_{k}.
\end{align*}
Let $f$ be a natural homomorphism from $\mathcal{R}^{(S)}[0]\to\mathcal{R}^{(V)}[0]$.
Put
\[
b_{0}=A\det(-1+\prod_{s=j}^{r}g_{v_{s}}(x_{\sigma(i)}))_{i,j=1}^{r}\in\mathcal{R}_{-1,0}^{(V)}
\]
and 
\[
b:=(b_{0},0,\dots,0)\in\prod_{i=0}^{r+1}\mathcal{R}_{i-1,i}^{(V)}=\mathcal{R}^{(V)}[1].
\]
Then $f(y_{k})=d(b)$. Therefore we have 
\[
\eta_{2}(\eta_{1}(\vartheta^{\diamond}))=b_{0}
\]
from the definition. By using Lemma \ref{lem:detred}, we have
\[
b_{0}\equiv A\det(g_{v_{i}}(u_{j})-1)_{1\leq i,j\leq r}\pmod{(I_{F^{\times}}\mathcal{R}(\emptyset)+J_{v_{0}}+\cdots+J_{v_{r}})\prod_{v\in V}J_{v}}.
\]
Thus we have
\begin{equation}
Q({\rm Sh}^{\diamond}|_{V})\equiv A\det(g_{v_{i}}(u_{j})-1)_{1\leq i,j\leq r}\pmod{(I_{F^{\times}}\mathcal{R}(\emptyset)+J_{v_{0}}+\cdots+J_{v_{r}})\prod_{v\in V}J_{v}}.\label{eq:Qdiamond_formula}
\end{equation}
Let $\bar{A}\in\mathbb{Z}[N^{S}/F^{\times}]=\mathbb{Z}[{\rm Gal}(H/F)]$
be the image of $A$. From (\ref{eq:Qdiamond_formula}), we have
\[
{\rm ch}(\mathfrak{q})^{[F:\mathbb{Q}]}\Theta_{S,\{\mathfrak{q}\},H}(s)\equiv s^{r}\det(-\log\left|u_{j}\right|_{v_{i}})_{i,j=1}^{r}\bar{A}\pmod{s^{r+1}}.
\]
Therefore, from Lemma \ref{lem:classical_cnf}, we have 
\[
\bar{A}=n_{S,T}\sum_{c\in N^{S}/F^{\times}}[c]
\]
Hence, from (\ref{eq:Qdiamond_formula}), we have 
\begin{align*}
Q^{N}({\rm Sh}^{\diamond}|_{V}) & \equiv{\rm ch}(\mathfrak{q})^{[F:\mathbb{Q}]}n_{S,T}\sum_{c\in N^{S}/F^{\times}}[c]\det(f_{v_{i}}(u_{j})-1)_{1\leq i,j\leq r}\pmod{I_{H}I_{v_{1}}\cdots I_{v_{r}}}\\
 & ={\rm ch}(\mathfrak{q})^{[F:\mathbb{Q}]}\hat{R}_{\mathfrak{q}}.
\end{align*}
\end{proof}
\begin{lem}
\label{lem:st_7}For $\mathfrak{q}\in Y$, we have
\[
{\rm ch}(\mathfrak{q})^{[F:\mathbb{Q}]}\Theta_{S,\{\mathfrak{q}\},K}\equiv{\rm ch}(\mathfrak{q})^{[F:\mathbb{Q}]}R_{G,S,\{\mathfrak{q}\}}\pmod{I_{G_{H}}\prod_{v\in V}I_{G_{v}}}.
\]
\end{lem}
\begin{proof}
By applying ${\rm rec}$ to the both hand sides of Lemma \ref{lem:calc_QNSH_reg},
we obtain the claim.\end{proof}
\begin{lem}
\label{lem:st_8}Let $\mathfrak{q}$ be an element of $T$ such that
${\rm ch}(\mathfrak{q})\neq2$. Then we have
\[
{\rm ch}(\mathfrak{q})^{[F:\mathbb{Q}]}\Theta_{S,T,K}\equiv{\rm ch}(\mathfrak{q})^{[F:\mathbb{Q}]}R_{G,S,T}\pmod{I_{G_{H}}\prod_{v\in V}I_{G_{v}}}.
\]
\end{lem}
\begin{proof}
From the definition, we have 
\[
\Theta_{S,T,K}=\Theta_{S,\{\mathfrak{q}\},K}\prod_{\mathfrak{p}\in T\setminus\{\mathfrak{q}\}}(1-N(\mathfrak{q})\sigma_{\mathfrak{q}}^{-1})
\]
and
\begin{align*}
R_{G,S,T} & =R_{G,S,\{\mathfrak{q}\}}\prod_{\mathfrak{p}\in T\setminus\{\mathfrak{q}\}}(1-N(\mathfrak{q})).\\
 & =R_{G,S,\{\mathfrak{q}\}}\prod_{\mathfrak{p}\in T\setminus\{\mathfrak{q}\}}(1-N(\mathfrak{q})\sigma_{\mathfrak{q}}^{-1})\ \ \ \ (\text{since }\forall\sigma\in G,(1-[\sigma])R_{G,S,\{\mathfrak{q}\}}=0).
\end{align*}
Thus the claim follows from Lemma \ref{lem:st_7}.
\end{proof}
Let us prove Theorem \ref{Thm:main2}. Assume that $T$ satisfies
the condition 
\[
\ker(\mu_{K}\to\prod_{\mathfrak{p}\in T_{_{K}}}(O_{K}/\mathfrak{p})^{\times})=\{1\}.
\]
Lemma \ref{lem:st_8} implies that
\[
{\rm ch}(\mathfrak{q})^{[F:\mathbb{Q}]}\Theta_{S,T,K}\equiv{\rm ch}(\mathfrak{q})^{[F:\mathbb{Q}]}R_{G,S,T}\pmod{I_{G_{H}}\prod_{v\in V}I_{G_{v}}}.
\]
From Lemma \ref{lem:vanish_2T} we have $2\Theta_{S,T,K}\in I_{G_{H}}\prod_{v\in V}I_{G_{v}}$.
From Lemma \ref{lem:2multiply} we have $2R_{G,S,T}\in I_{H}\prod_{v\in V}I_{G_{v}}$.
Thus we have
\[
\Theta_{S,T,K}\equiv R_{G,S,T}\pmod{I_{G_{H}}\prod_{v\in V}I_{G_{v}}},
\]
which completes the proof of Theorem \ref{Thm:main2}.

Let $\mathfrak{q}$ be prime ideal of $F$ such that ${\rm ch}(\mathfrak{q})\geq n+2$.
Since there exists a morphism from ${\rm Sh}^{v_{0}}$ to ${\rm Sh}^{v_{0},\diamond}$,
we have
\[
{\rm ch}(\mathfrak{q})^{n}Q({\rm Sh}^{v_{0}})\equiv Q({\rm Sh}^{v_{0},\diamond})
\]
Since there exists a morphism from ${\rm Sh}^{\diamond}$ to ${\rm Sh}^{v_{0},\diamond}$,
we have
\begin{align*}
Q({\rm Sh}^{v_{0},\diamond}) & \equiv Q({\rm Sh}^{\diamond})\\
 & \equiv{\rm ch}(\mathfrak{q})^{n}\hat{R}_{\mathfrak{q}}.
\end{align*}
Thus we have
\[
{\rm ch}(\mathfrak{q})^{n}Q({\rm Sh}^{v_{0}})\equiv{\rm ch}(\mathfrak{q})^{n}\hat{R}_{\mathfrak{q}}\pmod{I_{H}\prod_{v\in V}I_{v}}.
\]
Since $2Q({\rm Sh}^{v_{0}})$ and $2\hat{R}_{\mathfrak{q}}$ are in
$I_{H}\prod_{v\in V}I_{v}$ from Lemma \ref{lem:2multiply}, we have
\[
\hat{\Theta}_{S,\mathfrak{q},V}\equiv\hat{R}_{\mathfrak{q}},
\]
which completes the proof of Theorem \ref{thm:main_hat}.

\section{Some example}

In this section, we present a certain example of $\hat{\Theta}_{S,\mathfrak{q},S\setminus\{v_{0}\}}$
and Theorem \ref{thm:main_hat}. The reader who is not interested
in such an example can skip this section. 

Let us consider the case $F=\mathbb{Q}(\sqrt{5})$, $S=S_{\infty}$
and $\mathfrak{q}=(\sqrt{5})$. Let $v_{0}$ and $v_{1}$ be the infinite
places corresponding to the real embeddings $a+b\sqrt{5}\mapsto a+b\sqrt{5}$
and $a+b\sqrt{5}\mapsto a-b\sqrt{5}$ respectively. Then we have ${\rm Sh}^{v_{0}}=(\mathcal{B},\mathcal{L},\vartheta,1)$
where we put
\begin{eqnarray*}
\mathcal{B} & := & \mathcal{B}_{\emptyset,\mathbb{F}_{5}^{\times}}\mid_{{\bf Sub}(S\setminus\{v_{0})}\\
\mathcal{L} & := & \mathcal{L}_{\emptyset,\mathbb{F}_{5}^{\times}}\\
\vartheta & := & \vartheta_{\emptyset,\mathbb{F}_{5}^{\times}.}
\end{eqnarray*}
Note that $\mathcal{B}(W)=\mathcal{B}_{\emptyset,\mathbb{F}_{5}^{\times}}(W)$
is the certain $\mathbb{Z}[F^{\times}]$ submodule of
\[
\mathcal{S}(\mathbb{A}_{F}^{S_{\infty}})\otimes\mathcal{Z}(F^{\times},W)\otimes\mathbb{Z}[N^{S}]
\]
defined in Section \ref{sec:constSh}. We define $f\in\mathcal{S}(\mathbb{A}_{F}^{S_{\infty}})$
by $f(x):=\prod_{v}f_{v}(x_{v})$ where $f_{v}=\bm{1}_{O_{v}}$ for
$v\neq\mathfrak{q}$ and 
\[
f_{\mathfrak{q}}(x_{\mathfrak{q}})=\begin{cases}
1 & x_{\mathfrak{q}}\in O_{\mathfrak{q}}^{\times}\\
-4 & x_{\mathfrak{q}}\in\mathfrak{q}O_{\mathfrak{q}}\\
0 & x_{\mathfrak{q}}\notin\mathfrak{q}O_{\mathfrak{q}}.
\end{cases}
\]
We put $\epsilon=\frac{1+\sqrt{5}}{2}$ and 
\begin{eqnarray*}
D_{0} & := & [1]+[1,\epsilon^{2}]\in\mathcal{Z}(F^{\times},\emptyset)\\
D_{1} & := & [1,\epsilon]\in\mathcal{Z}(F^{\times},\{v_{1}\}).
\end{eqnarray*}
Then we have $f\otimes D_{0}\otimes[1]\in\mathcal{B}(\emptyset)$,
and $\vartheta\in H_{0}(F^{\times},\mathcal{B}(\emptyset))$ is the
image of $f\otimes D_{0}\otimes[1]$. Define $a=(a_{0},a_{1})\in\mathcal{B}_{0,0}\oplus\mathcal{B}_{1,1}=\mathcal{B}[0]$
by
\begin{eqnarray*}
a_{0} & := & f\otimes D_{0}\otimes[1]\in\mathcal{B}(\emptyset)\\
a_{1} & := & [1,\epsilon]\otimes\left(f\otimes D_{1}\otimes[1]\right)\in\mathcal{I}_{1}\otimes\mathcal{B}(\{v_{1}\}).
\end{eqnarray*}
Then we have
\[
d_{1}(a)=(\vartheta,0)\in\mathcal{B}_{0,-1}\oplus\mathcal{B}_{1,0}.
\]

For $W\subset S_{\infty}$, we denote the element $(x_{v})_{v\in S\setminus W}\in N_{S\setminus W}$
by
\[
(x_{0},x_{1})\ \ \ \ (x_{j}\in N_{v_{j}}\sqcup\{*\})
\]
where $x_{i}=x_{v_{i}}$ for $v_{i}\in S\setminus W$ and $x_{i}=*$
for $v_{i}\in W$. We have 
\begin{eqnarray*}
\mathcal{L}(a_{0}) & = & Z_{0}(s_{0},s_{1})\otimes[(+1,+1)]\\
\mathcal{L}(a_{1}) & = & [1,\epsilon]\otimes\left(Z_{1}(s_{0})\otimes[(+1,*)]\right)
\end{eqnarray*}
where $Z_{0}(s_{0},s_{1})$ and $Z_{1}(s_{0})$ are Maclaurin series
of
\[
\sum_{x\in C(1,\epsilon^{2})\cup C(\epsilon^{2})}f(x)|x|_{v_{0}}^{-s_{0}}|x|_{v_{1}}^{-s_{1}}
\]
and
\[
\sum_{x\in C(1,\epsilon)}f(x)|x|_{v_{0}}^{-s_{0}}
\]
respectively. Define $b=(b_{0},b_{1})\in\mathcal{B}_{-1,0}\oplus\mathcal{B}_{0,1}=\mathcal{B}[1]$
by
\begin{eqnarray*}
b_{1} & = & [1,\epsilon]\otimes\left(Z_{1}(s_{0})\otimes[(+1,+1)]\right)\\
b_{0} & = & \left(Z_{0}(s_{0},s_{1})-Z_{1}(s_{0})\right)\otimes[(+1,+1)]+|\epsilon|_{v_{0}}^{-s_{0}}|\epsilon|_{v_{1}}^{-s_{1}}Z_{1}(s_{0})\otimes[(+1,-1)].
\end{eqnarray*}
Since $d(b)=a$, we have 
\[
Q({\rm Sh}^{v_{0}})=b_{0}.
\]
Since the constant term of $Z_{0}(s_{0},s_{1})$ (resp. $Z_{1}(s_{0})$)
is equal to $0$ (resp. $-1$), we have
\[
\hat{\Theta}_{S,\mathfrak{q},\{v_{1}\}}=Q^{N}({\rm Sh}^{v_{0}})=[(+1,+1)]-[(+1,-1)]\in I_{v_{1}}/(I_{F^{\times}}I_{v_{1}}).
\]
Note that we have $H=F$ and $n_{S,T}=-1$. Since $u=-\epsilon^{-2}$
is a generator of $\mathcal{O}_{S,\{\mathfrak{q}\}}^{\times}$ such
that $-(-\log|u|_{v_{1}})>0$, we have
\begin{align*}
\hat{R}_{\mathfrak{q}} & =-([1,u]-[1,1])\\
 & =[(+1,+1)]-[(+1,-1)]\in I_{v_{1}}/I_{H}I_{v_{1}}.
\end{align*}
Thus we have
\[
\hat{\Theta}_{S,\mathfrak{q},\{v_{1}\}}\equiv\hat{R}_{\mathfrak{q}}\pmod{I_{H}I_{v_{1}}}.
\]
Theorem \ref{thm:main_hat} says that such congruences holds in more
general settings. Note that Theorem \ref{Thm:main2} says nothing
in this case because there exists no non-trivial extension $K$ of
$F$ unramified outside $S_{\infty}$. 

\bibliographystyle{plain}
\bibliography{GrossConjecture}

\begin{thebibliography}{1}

\bibitem{MR1104783}
Noboru Aoki.
\newblock Gross' conjecture on the special values of abelian {$L$}-functions at
  {$s=0$}.
\newblock {\em Comment. Math. Univ. St. Paul.}, 40(1):101--124, 1991.

\bibitem{MR2336038}
David Burns.
\newblock Congruences between derivatives of abelian {$L$}-functions at
  {$s=0$}.
\newblock {\em Invent. Math.}, 169(3):451--499, 2007.

\bibitem{MR3351752}
Pierre Charollois, Samit Dasgupta, and Matthew Greenberg.
\newblock Integral {E}isenstein cocycles on {$\bold{GL}_n$}, {II}: {S}hintani's
  method.
\newblock {\em Comment. Math. Helv.}, 90(2):435--477, 2015.

\bibitem{DasguptaSpiess}
Samit Dasgupta and Michael Spiess.
\newblock Partial zeta values, {G}ross's tower of fields conjecture, and
  {G}ross-{S}tark units.
\newblock {\em preprint}.

\bibitem{MR3198753}
Francisco Diaz~y Diaz and Eduardo Friedman.
\newblock Signed fundamental domains for totally real number fields.
\newblock {\em Proc. Lond. Math. Soc. (3)}, 108(4):965--988, 2014.

\bibitem{MR931448}
Benedict~H. Gross.
\newblock On the values of abelian {$L$}-functions at {$s=0$}.
\newblock {\em J. Fac. Sci. Univ. Tokyo Sect. IA Math.}, 35(1):177--197, 1988.

\bibitem{MR2392823}
Richard Hill.
\newblock Shintani cocycles on {${\rm GL}_n$}.
\newblock {\em Bull. Lond. Math. Soc.}, 39(6):993--1004, 2007.

\end{thebibliography}

\end{document}